\documentclass[11pt,a4paper]{amsart}
\usepackage[margin=1in]{geometry}
\usepackage[utf8]{inputenc}
\usepackage{amsmath}
\usepackage{amsfonts}
\usepackage{xypic}
\usepackage{amsthm}
\usepackage{mathrsfs}
\usepackage{amssymb}
\usepackage[all]{xy}
\usepackage{color}

\usepackage{hyperref}
\usepackage{graphicx}

\numberwithin{equation}{section}
\newtheorem{theorem}{Theorem}[section]
\newtheorem{lemma}[theorem]{Lemma}
\newtheorem{proposition}[theorem]{Proposition}

\theoremstyle{remark}
\newtheorem{remark}[theorem]{Remark}

\newtheorem{innercustomgeneric}{\customgenericname}
\providecommand{\customgenericname}{}
\newcommand{\newcustomtheorem}[2]{%
  \newenvironment{#1}[1]
  {%
   \renewcommand\customgenericname{#2}%
   \renewcommand\theinnercustomgeneric{##1}%
   \innercustomgeneric
  }
  {\endinnercustomgeneric}
}

\newcustomtheorem{customthm}{Theorem}
\newcustomtheorem{customprop}{Proposition}
\newcustomtheorem{customlemma}{Lemma}

\newtheoremstyle{TheoremNum}
        {\topsep}{\topsep}              
        {\itshape}                      
        {}                              
        {\bfseries}                     
        {.}                             
        { }                             
        {\thmname{#1}\thmnote{ \bfseries #3}}
    \theoremstyle{TheoremNum}

\DeclareMathOperator{\Hom}{Hom}
\DeclareMathOperator{\Aut}{Aut}

\newcommand{\Spec}{\mathrm{Spec}}

\newcommand{\Obs}{\mathrm{Obs}}

\newcommand{\End}{\mathrm{End}}
\newcommand{\Ext}{\mathrm{Ext}}

\newcommand{\ad}{\mathrm{ad}}

\newcommand{\id}{\mathbf{1}}

\newcommand{\image}{\mathrm{Im}}

\newcommand{\Art}{\mathbf{(Art)}}

\newcommand{\FinVect}{\mathbf{(FinVect)}}
\newcommand{\Sets}{\mathbf{(Sets)}}
\newcommand{\alg}{\mathrm{alg}}
\newcommand{\Def}{\mathrm{Def}}
\newcommand{\Coh}{\mathrm{Coh}}

\newcommand{\Ort}{\mathrm{O}}
\newcommand{\Sp}{\mathrm{Sp}}

\newcommand{\Dd}{\mathcal{D}}
\newcommand{\Ee}{\mathcal{E}}
\newcommand{\Ff}{\mathcal{F}}

\newcommand{\Mm}{\mathcal{M}}

\newcommand{\Oo}{\mathcal{O}}

\newcommand{\Vv}{\mathcal{V}}
\newcommand{\Ww}{\mathcal{W}}

\newcommand{\mM}{\mathfrak{m}}

\newcommand{\D}{\mathbf{D}}
\newcommand{\F}{\mathbf{F}}
\renewcommand{\H}{\mathrm{H}}

\renewcommand{\S}{\mathbf{S}}

\newcommand{\HH}{\mathbb{H}}

\newcommand{\quotient}[2]{{\raisebox{.2em}{\thinspace $#1$}\left / \raisebox{-.15em}{ $#2$}\right.}}

\newcommand\Quotient[2]{
        \mathchoice
            {
                \text{\raise1ex\hbox{\thinspace $#1$}\Big{/} \lower1ex\hbox{$#2$} \thinspace}%
            }
            {
                #1\,/\,#2
            }
            {
                #1\,/\,#2
            }
            {
                #1\,/\,#2
            }
    }

\newcommand\GIT[2]{
        \mathchoice
            {
                \text{\raise1ex\hbox{\thinspace $#1$}\Big{/}\!\!\!\!\Big{/} \lower1ex\hbox{$#2$} \thinspace}%
            }
            {
                #1\,/\,#2
            }
            {
                #1\,/\,#2
            }
            {
                #1\,/\,#2
     a       }
    }

\newcommand{\morph}[6]{\begin{array}{cccc} #6: & #1  & \stackrel{#5}{\longrightarrow} &  #2  \\ & #3 & \longmapsto & #4  \end{array}}

\title{\bf Deformation theory of orthogonal and symplectic sheaves}

\author[E. Franco]{Emilio Franco}
\address{E. Franco,
\newline\indent Centro de An\'alise Matem\'atica, Geometria e Sistemas Din\^{a}micos, 
\newline\indent Instituto Superior T\'ecnico, Universidade de Lisboa, 
\newline\indent Av. Rovisco Pais s/n, 1049-001 Lisboa, Portugal}
\email{emilio.franco@tecnico.ulisboa.pt}

\date{\today}

\thanks{The author is supported by the Scientific Employment Stimulus program, fellowship reference CEECIND/04153/2017, funded by FCT (Portugal) with national funds. He aknowledges ICMAT (Madrid) for the financial support during a research stay at the centre.}

\begin{document}

\begin{abstract}
We show that the space of first-order deformations of an orthogonal (resp. symplectic) sheaf over a smooth projective scheme is the first hypercohomology space of a complex which is naturally constructed out of the orthogonal (resp. symplectic) sheaf. We also provide an obstruction theory of these objects whose target is the second hypercohomology space of this complex.
\end{abstract}

\maketitle


\section{Introduction}

Moduli spaces of principal bundles usually carry interesting geometric structures, being a powerful, and often unique, source of examples of varieties with prescribed properties and characteristics. Nevertheless, these spaces might be non-compact whenever the base (smooth) scheme has dimension higher than $1$. Principal sheaves or singular principal
bundles \cite{gomez&sols_orthogonal, gomez&sols_annals,Sch1,Sch2,GLSS1,GLSS2} provide a natural compactification of the moduli space of principal bundles for a connected complex reductive structure group. Therefore, moduli spaces of principal sheaves are projective varieties equipped with an interesting geometry, at least, on a dense subset. In order to check whether or not these properties extend to the compactification, we need a local description of the moduli spaces, precisely over the locus where the principal sheaves fail to be principal bundles. Such description would naturally derive from deformation theory of principal sheaves, which is still missing at present date.

Moduli spaces of sheaves have been very useful in defining invariants. For
instance, Donaldson polynomials or, more recently, Donaldson-Thomas
invariants. It is natural to consider also $G$-principal objects
for a reductive group $G$ but, again, we first need to study the deformation
theory of these objects.

In this article we consider orthogonal and symplectic sheaves (see Section \ref{sc orthogonal and symplectic sheaves} for the definition).
We show that the deformation and obstruction theory of these objects is controlled by a deformation complex naturally built out of our starting orthogonal (resp. symplectic) sheaf. A close version of this deformation complex appears in \cite{scalise} where a preliminary study of the deformation theory of quadratic sheaves is presented, along with a beautiful study of framed symplectic sheaves and their moduli spaces (see also \cite{scalise2}).


Let us briefly sketch the structure of our paper. After recalling the basic definitions of deformation and obstruction theory in \ref{sc deformation and obstruction theory}, we review in Section \ref{sc coherent sheaves} the classical case of coherent sheaves achieved by Grothendieck. We finish the preliminaries by presenting orthogonal and symplectic sheaves in Section \ref{sc orthogonal and symplectic sheaves}. In Section \ref{sc def complex} we introduce the deformation complex, providing a description of its zero, first and second hypercohomology spaces. We prove in Section \ref{sc deformation} that the space of first order deformations of orthogonal (resp. symplectic) sheaves coincide with the first hypercohomology space and we construct, in Section \ref{sc obstruction}, an obstruction theory for these objects with the second hypercohomology space as a target.

\subsubsection*{Acknowledgments}

The author wishes to thank Tom\'as L. G\'omez for introducing him on the subject of deformation of $G$-sheaves, for his mathematical insights on the problem, and, in general, for his generous support and his invaluable help during the completion of this article.

This project started during a research stay of the author at ICMAT (Madrid) and he expresses his gratitude for the warm hospitality he received. Special thanks go to O. Garc\'ia-Prada for making possible this stay.

\section{Preliminaries}

\subsection{Deformation and obstruction theory}
\label{sc deformation and obstruction theory}

See \cite{nitsure} for an introduction to deformation theory.
Let $k$ be an algebraically closed field, $\Art$ the category of all finite Artin local $k$-algebras with residue field $k$ and denote by $\Sets$ the category of all sets. Denote by $\FinVect$ the category of finite dimensional $k$-vector spaces. We construct the functor $k\langle \bullet \rangle : \FinVect \to \Art$ by setting $k\langle V \rangle = k \oplus V$ as $k$-vector spaces, and ring structure given by $(k,v)\cdot (k', v') = (kk', k'v + kv')$. Note that $k\langle V \rangle$ is the Artin local algebra whose maximal ideal is the vector space $\mM = V$, satisfying $\mM^2 = 0$, and its residue field is $k$. Note that one naturally has that $k\langle k \rangle \cong k[\epsilon]/(\epsilon^2)$. We say that $0 \to H \to B \stackrel{\tau}{\to} A \to 0$ is a {\em small extension} in $\Art$ if $\mM_B H = 0$.

Given a deformation functor $\F: \Art \to \Sets$, we list below the so-called {\em Schlessinger conditions} for $\F$.

\begin{itemize}
\item[\textbf{S1}:] For any homomorphism $C \to A$ and any small extension $0 \to H \to B \stackrel{\tau}{\to} A \to 0$ in $\Art$, the induced morphism 
\[
\F(B \times_A C) \longrightarrow \F(B) \times_{\F(A)} \F(C)
\]
is surjective.
 
\item[\textbf{S2}:] For any $B \in \Art$ and any $V \in \FinVect$, the induced morphism
\[
\F(B \times_k k\langle V \rangle) \longrightarrow \F(B) \times \F(k \langle V \rangle)
\]
is bijective.

\item[\textbf{S3}:] The space of first-order deformations $\F(k[\epsilon]/(\epsilon^2))$ is a finite dimensional $k$-vector space.

\item[\textbf{S4}:] For any $0 \to H \to B \stackrel{\tau}{\to} A \to 0$ small extension in $\Art$, the induced morphism
\[
\F(B \times_A B) \longrightarrow \F(B) \times_{\F(A)} \F(B)
\]
is bijective.
\end{itemize}

A pro-family is a family $r$ parametrized by a complete local $k$-algebra
$R$ with residue field $k$. A pro-family is versal if any family parametrized
by $A\in \Art$ is the pull-back of $r$ by a morphism $f:R \to A$.
It is a miniversal pro-family if furthermore
the induced map on first order infinitesimal deformations
$$
\Hom_{k-\alg}\left ( R, k[\epsilon]/(\epsilon^2) \right ) \longrightarrow \F \left ( k[\epsilon]/(\epsilon^2) \right ) 
$$
is an isomorphism. A functor is pro-representable if there is a
universal pro-family (i.e, the morphism $f:R \to A$ above is unique).

Schlessinger \cite{schlessinger} proved that a deformation functor admits a miniversal pro-family if and only if it satisfies \textnormal{\textbf{S1}}, \textnormal{\textbf{S2}} and \textnormal{\textbf{S3}}. Moreover, it is pro-representable if and only if it satisfies \textnormal{\textbf{S4}} along with the previous conditions.

An {\em obstruction theory} for the deformation functor $\F$ consists on a $k$-vector space $\Obs(\F)$ and, for any small extension $0 \to H \to B \stackrel{\tau}{\to} A \to 0$ in $\Art$, a morphism
\begin{equation} \label{eq obs morphism}
\Omega_\tau : \F(A) \longrightarrow  H \otimes_k \Obs(\F)
\end{equation}
satisfying the conditions listed below:

\begin{itemize}
\item[\textbf{O1}:] The sequence of sets 
\[
\xymatrix{
\F(B) \ar[r]^{\F(\tau)} & \F(A) \ar[r]^{\Omega_\tau \qquad}  & H \otimes_k \Obs(\F)
}
\]
is exact in the middle.
 
\item[\textbf{O2}:] For any morphism of small extensions, that is, a commutative diagram
\[
\xymatrix{
0 \ar[r] & H \ar[r] \ar[d]_{h} & B \ar[r]^{\tau} \ar[d]_{\beta} & A\ar[r] \ar[d]^{\alpha} & 0
\\
0 \ar[r] & H' \ar[r] & B' \ar[r]^{\tau'} & A' \ar[r] & 0,
}
\]
the induced diagram, 
\[
\xymatrix{
\F(A) \ar[r]^{\Omega_\tau \qquad} \ar[d]_{\F(\alpha)} &  H \otimes_k \Obs(\F) \ar[d]^{\id_{\Obs} \otimes_k h}
\\
\F(A') \ar[r]^{\Omega_{\tau'} \qquad} &  H' \otimes_k \Obs(\F),
}
\]
commutes.
\end{itemize}


\subsection{Coherent sheaves}
\label{sc coherent sheaves}

Let $X$ be a projective scheme over $k$. Denote by $\Coh(X)$ the category of coherent sheaves on $X$, and write $\Dd^b(X)$ for the bounded derived category of quasi-coherent sheaves with coherent cohomology. Given $E \in \Coh(E)$, by abuse of notation, we denote by $E \in \Dd^b(X)$ the complex supported on $0$-degree given by $E$. 

For any coherent sheaf $E$ one defines its {\em dual sheaf} by setting $E^\vee := \Hom_{\Coh(X)}(E, \Oo_X)$ and, for any complex $F^\bullet \in \Dd^b(X)$, its associated {\em dual complex} is $\D F^\bullet := \Hom_{\Dd^b(X)}(F^\bullet, \Oo_X)$. Every torsion free coherent sheaf injects naturally into its double dual, $E \hookrightarrow E^{\vee \vee}$, but unless $E$ is reflexive, $E^{\vee \vee}$ is not isomorphic to $E$. On the other hand $\D \circ \D = \id$, so $\D$ is an autoequivalence of $\Dd^b(X)$. If $E$ is locally free, $\D E \cong E^\vee$, but this does not hold for a general coherent sheaf. 
For any coherent sheaf $E$ over a projective scheme $X$, we define its deformation functor
\[
\Def_E : \Art \longrightarrow \Sets
\]
by associating to any $A \in \Art$ the set of isomorphism classes of pairs $(\Ee, \gamma)$, where $\Ee$ is a coherent sheaf on $X_A : = X \times \Spec(A)$, flat over $\Spec(A)$, and $\gamma: \Ee|_X \to E$ is an isomorphism. We say that two pairs $(\Ee, \gamma)$ and $(\Ee', \gamma')$ are isomorphic if there exists an isomorphism $f : \Ee \to \Ee'$ such that $\gamma' \circ (f|_X) = \gamma$. Every morphism of Artin algebras $a : A \to A'$, induces naturally a morphism $p_a : X_{A'} \to X_A$. Functoriality of $\Def_E$ follows from applying pull-backs under $p_a$.

Grothendieck showed that the cohomology of the complex $\D E \otimes^L E$ rules the deformation and obstruction theory of $E$. In particular $\Def_E$ admits a miniversal pro-family and its space of first-order deformations is 
\[
\Def_E(k[\epsilon]/(\epsilon^2)) \cong \Ext^1_X(E,E) \cong \H^1 (\D E \otimes^L E).
\]
If further, $E$ is simple, then $\Def_E$ is pro-representable. Also, $\Def_E$ admits a deformation theory with vector space
\[
\Obs(\Def_E) \cong \Ext^2_X(E,E) \cong \H^2 (\D E \otimes^L E). 
\]
In particular, when $\Ext^2_X(E,E) = 0$, the deformation functor $\Def_E$ is formally smooth.

\subsection{Orthogonal and symplectic sheaves}
\label{sc orthogonal and symplectic sheaves}

An {\em orthogonal sheaf} ({\em resp.} a {\em symplectic sheaf}) on the projective scheme $X$ is a pair $(E, \phi)$, where $E$ is a torsion-free coherent sheaf on $X$ and $\phi : E \otimes E \to \Oo_X$ is a homomorphism which is symmetric, $\phi \circ \theta_E = \phi$ (resp. anti-symmetric, $\phi \circ \theta_E = \phi$), under the permutation $\theta_E : E \otimes E \to E \otimes E$, and such that its restriction $\phi|_{U_E}$ to the open subset $U_E$ where $E$ is locally free is non-degenerate. Recalling that the centre of $\Ort(n,k)$ ({\em resp.} $\Sp(2m,k)$) is $\{ \id, -\id\}$, we say that an orthogonal ({\em resp.} symplectic) sheaf $(E,\phi)$ is {\em simple} if its automorphism group is $\Aut(E,\phi) = \{ \id_E, -\id_E \}$.

A {\em family of orthogonal} (resp. {\em symplectic}) {\em sheaves} parametrized by $S$ is a pair $(\Ee, \Phi)$ such that $\Ee \to X \times S$ is a torsion-free coherent sheaf, flat over $S$ and such that $\Ee_s := \Ee|_{X \times \{ s \}}$ is torsion-free for each closed point $s \in S$, and $\Phi : \Ee \otimes \Ee \to \Oo_{X \times S}$ is a symmetric (resp. anti-symmetric) homomorphism whose restriction $\Phi|_{U_\Ee}$ to the open set where $\Ee$ is locally free is non-degenerate. Two families $(\Ee, \Phi)$ and $(\Ee', \Phi')$ of orthogonal (resp. symplectic) sheaves are {\em isomorphic} if there exists an isomorphism $f : \Ee \to \Ee'$ such that $\Phi|_{U_\Ee} = \Phi'|_{U_\Ee} \circ (f|_{U_\Ee} \otimes f|_{U_\Ee})$.

For any coherent sheaf $E$, adjunction gives the identification 
\begin{equation} \label{eq adjunction in Coh}
\begin{array}{ccc}   \Hom_{\Coh(X)}(E \otimes E, \Oo_X))  & \cong  &  \Hom_{\Coh(X)}(E, E^\vee))  
\\
\phi & \longmapsto & \phi_\ad.  
\end{array}
\end{equation}
Permutation composed with the adjunction is the same as restricting to $E \hookrightarrow E^{\vee \vee}$ the dual to the adjoint morphism, 
\begin{equation} \label{eq phi_ad and permutation}
(\phi \circ \theta_E)_\ad = \phi_\ad^\vee|_E : E \hookrightarrow E^{\vee \vee}  \stackrel{\phi_\ad^\vee}{\longrightarrow} E^\vee.
\end{equation}

\begin{remark}
It follows from \eqref{eq adjunction in Coh} and \eqref{eq phi_ad and permutation} that there exists a 1:1 correspondence between orthogonal (resp. symplectic) sheaves and pairs $(E, \phi_\ad)$, where $E$ is a torsion-free sheaf, and $\phi_\ad : E \to E^\vee$ is such that $\phi_\ad = \phi_\ad^\vee|_E$ (resp. $\phi_\ad = -\phi_\ad^\vee|_E$), and $\phi_\ad|_{U_E}$ is an isomorphism.
\end{remark}

As in the case of coherent sheaves, the deformation theory of an orthogonal ({\em resp.} symplectic) sheaf is related to a complex in the derived category naturally build out of it. In the remaining of the section we will see how orthogonal and symplectic sheaves provide a well defined geometrical object living in the derived category.

Given any complex $F^\bullet \in \Dd^b(X)$, denote by $\theta_{F^\bullet}$ the derived permutation of $F^\bullet \otimes^L F^\bullet$. Adjunction gives 
\begin{equation} \label{eq adjunction in D}
\begin{array}{ccc}   \Hom_{\Dd^b(X)}(F^\bullet \otimes F^\bullet, \Oo_X))  & \cong  &  \Hom_{\Dd^b(X)}(F^\bullet, \D F^\bullet))  
\\
\phi^\bullet & \longmapsto & \phi^\bullet_\ad.
\end{array}
\end{equation}
As in the case of sheaves, permutation composed with adjunction dualizes morphisms of the form $\psi^\bullet : F^\bullet \to \D F^\bullet$,
\begin{equation} \label{eq theta and duality}
(\psi^\bullet \circ \theta_{F^\bullet})_\ad = \D \psi^\bullet_\ad.
\end{equation}

One can give a description of orthogonal and symplectic sheaves in terms of the derived category.

\begin{proposition} \label{pr derived quadratic sheaves}
There exists a 1:1 correspondence between orthogonal (resp. symplectic) sheaves and pairs $(E, \phi_\ad)$, where $E$ is a complex supported on $0$ determined by a torsion-free sheaf, and $\phi_\ad\in \Hom_{\Dd^b(X)}(E, \D E)$ such that $\phi_\ad= \D \phi_\ad$ (resp. $\phi_\ad= - \D \phi_\ad$) and $\HH^0(\phi_\ad)|_{U_E} : E|_{U_E} \to E^\vee|_{U_E}$ is an isomorphism.
\end{proposition}

\begin{proof}
This follows immediately after \eqref{eq adjunction in Coh}, \eqref{eq adjunction in D}, \eqref{eq theta and duality} and the identification
\[
\Hom_{\Dd^b(X)}(E, \D E) \cong \Hom_{\Coh(X)}(E, E^\vee),
\]
which can be proved by applying the standard truncations on $\D E$.
\end{proof}

\section{The deformation complex}

\label{sc def complex}

Consider a coherent sheaf $E$ over $X$ projective and a morphism $\phi : E \otimes E \to \Oo_X$ such that $\phi \circ \theta_{E} = \phi$. Inspired by J. Scalise \cite{scalise}, we define 
\[
\Delta_{(E,\phi)}^+ := \left ( \id_{\D E \otimes^L \D E} + \theta_{\, \D E} \right ) \circ (\id_{\D E} \otimes \phi_\ad) : \D E \otimes^L E \longrightarrow \left ( \D E \otimes^L \D E \right )^{+}.
\]
Analogously, when $\phi \circ \theta_{E} = -\phi$, set
\[
\Delta_{(E,\phi)}^- :=  \left ( \id_{\D E \otimes^L \D E} - \theta_{\, \D E} \right )  \circ (\id_{\D E} \otimes \phi_\ad) : \D E \otimes^L E \longrightarrow \left ( \D E \otimes^L \D E \right )^{-}.
\]
Let us also consider their associated mapping cone shifted by $-1$,
\[
\Delta^{+,\bullet}_{(E,\phi)}: = C^\bullet(\Delta^{+}_{(E,\phi)})[-1],
\]
and
\[
\Delta^{-,\bullet}_{(E,\phi)}: = C^\bullet(\Delta^{-}_{(E,\phi)})[-1].
\]
These complexes fit in the distinguished triangle
\[
\Delta^{\pm,\bullet}_{(E,\phi)}\longrightarrow \D E \otimes^L E  \stackrel{\Delta^{\pm}_{(E,\phi)}}{\longrightarrow} \left ( \D E \otimes^L \D E \right )^{\pm} \longrightarrow \Delta^{\pm,\bullet}_{(E,\phi)}[1],
\]
giving a long-exact sequence in hypercohomology. 

\begin{proposition} \label{pr description of HH^0}
Let $E$ be a coherent sheaf over $X$ projective and consider $\phi : E \otimes E \to \Oo_X$ such that $\phi \circ \theta_E = \pm \phi$. Consider a finite dimensional $k$-vector space $V$. Then,
\[
V \otimes_k \HH^0\left (\Delta^{\pm,\bullet}_{(E,\phi)}\right ) = \left \{ \lambda \in \Hom_{\Coh}(E,E) \textnormal{ such that } \phi \circ (\id_E \otimes \lambda) + \phi \circ (\lambda \otimes \id_E) = 0 \right \},  
\]
\end{proposition}

\begin{proof}
By definition $\HH^{i+1}\left (\Delta^{\pm,\bullet}_{(E,\phi)}\right ) = \HH^{i} \left (C^\bullet \left (\Delta^{\pm}_{(E,\phi)} \right ) \right )$, so we focus on the description of the complex $C^\bullet \left (\Delta^{\pm}_{(E,\phi)} \right )$. Picking the locally free resolution $W^\bullet \stackrel{\pi}{\longrightarrow} E \to 0$, one can see that this complex is $0$ for $H < -1$, and in degrees $-1$ and $0$ amounts to
\[
0 \longrightarrow \Hom_X(W_0, E) \stackrel{\delta_{-1}}{\longrightarrow} \begin{pmatrix} \Hom_X(W_{-1}, E) \\ \Hom_X(W_0 \otimes W_0, \Oo_X)^{\, \pm} \end{pmatrix}, 
\]
with
\[
\delta_{-1} = \begin{pmatrix}-(\bullet) \circ \partial_{-1} \\ \phi \circ (\pi \otimes (\bullet)) \circ (\id_{W_0\otimes W_0} \pm \theta_{W_0} ) \end{pmatrix},
\]
where $\pi$ denotes the projection $W_0 \stackrel{\pi}{\longrightarrow} E \to 0$ and its dual, $\pi^\vee$, the inclusion $0 \to E^\vee \stackrel{\pi^\vee}{\longrightarrow} W_0^\vee$. The first statement follows from the fact that $\HH^{-1} \left (C^\bullet \left (\Delta^{\pm}_{(E,\phi)} \right ) \right ) = \ker (\delta_{-1})$ and 
\begin{align*}
\phi \circ (\pi \otimes \lambda) \pm \phi \circ (\pi \otimes \lambda) \circ \theta_{W_0} ) & = \phi \circ (\pi \otimes \lambda) \pm \phi  \circ \theta_{E} \circ (\lambda \otimes \pi) ) 
\\
& = \phi \circ (\pi \otimes \lambda) + \phi \circ (\lambda \otimes \pi) ).
\end{align*}
And the result follows.
\end{proof}

\begin{proposition} \label{pr description HH^1}
Let $E$ be a coherent sheaf over $X$ projective and consider $\phi : E \otimes E \to \Oo_X$ such that $\phi \circ \theta_E = \pm \phi$. Consider a finite dimensional $k$-vector space $V$. Then, $V \otimes_k \HH^1\left (\Delta^{\pm,\bullet}_{(E,\phi)}\right )$ is the finite dimensional vector space classifying isomorphism classes of $1$-extensions 
\begin{equation} \label{eq extension HH^1}
0 \longrightarrow V \otimes_k E \stackrel{i}{\longrightarrow} F \stackrel{j}{\longrightarrow} E \longrightarrow 0,
\end{equation}
equipped with 
\[
\Phi : \quotient{F \otimes F}{I} \to V \otimes_k \Oo_X,
\]
where $I \subset F \otimes F$ is the subsheaf generated by $\left (f_1, (i \circ j)(f_2) \right ) - \left ( (i \circ j)(f_1), f_2 \right )$ for all $f_i \in F$, and $\Phi$ is such that
\begin{equation} \label{eq Phi is pm}
\Phi \circ \theta_{F} = \pm \Phi
\end{equation}
and
\begin{equation} \label{eq Phi restricts to phi}
\Phi \circ (\id_F \otimes i) = (\id_V \otimes_k \phi) \circ \left ( \pi \otimes \id_{(V \otimes_k E)} \right ).
\end{equation}
\end{proposition}

\begin{proof} 
We have to describe $V \otimes_k \HH^0 \left (C^\bullet \left (\Delta^{\pm}_{(E,\phi)} \right ) \right )$, which is finite dimensional since $V$ is and so are all the cohomology spaces of $\D E \otimes^L E$ and $\D E \otimes^L \D E$. Taking the locally free resolution $W^\bullet \stackrel{\pi}{\longrightarrow} E \to 0$, the complex $C^\bullet \left (\Delta^{\pm}_{(E,\phi)} \right )$ can be described in degrees $0$ and $1$ as
\[
\begin{pmatrix} \Hom_X(W_{-1}, E) \\ \Hom_X(W_0 \otimes W_0, \Oo_X)^\pm \end{pmatrix} \stackrel{\delta_{0}}{\longrightarrow} \begin{pmatrix} \Hom_X(W_{-2}, E) \\ \Hom_X(W_0 \otimes W_{-1}, \Oo) \end{pmatrix},
\]
with
\[
\delta_0 = \begin{pmatrix} -(\bullet) \circ \partial_{-2} & 0 \\ \phi \circ ( \pi \otimes (\bullet)) & (\bullet) \circ (\id_{W_0} \otimes \partial_{-1}) \end{pmatrix}.
\] 
We have
\[
\ker (\delta_0) = \left \{ \begin{matrix} (\eta, \Psi) \in \Hom_X(W_{-1}, E) \oplus \Hom_X(W_0 \otimes W_0, \Oo_X)^\pm \\ \eta \circ \partial_{-2} = 0 \\ \phi \circ ( \pi \otimes \eta) + \Psi \circ (\id_{W_0} \otimes \partial_{-1}) = 0 \end{matrix} \right \}
\]
and $\HH^0 \left (C^\bullet \left (\Delta^{\pm}_{(E,\phi)} \right ) \right )=H^0(\ker(\delta_0))$.
For any pair $(\overline{\eta},\overline{\Psi}) \in V \otimes_k  H^0(\ker (\delta_0))$, with $\overline{\eta} \in \Hom_X(W_{-1}, V \otimes_k E)$ and $\overline{\Psi} \in \Hom_X(W_0 \otimes W_0, V \otimes_k \Oo_X)^\pm$, one can naturally construct an extension \eqref{eq extension HH^1} setting
\[
F := (V \otimes_k E) \oplus W_{0} / (\overline{\eta} \oplus \partial_{-1})(W_{-1}),
\]
with the injection
\[
\morph{E}{F = (V \otimes_k E) \oplus W_{0} / (\overline{\eta} \oplus \partial_{-1})(W_{-1})}{\overline{e}}{[(\overline{e}, 0)]}{}{i}
\]
and the projection
\[
\morph{F = (V \otimes_k E) \oplus W_{0} / (\overline{\eta} \oplus \partial_{-1})(W_{-1})}{W_{0} / \partial_{-1}(W_{-1}) \cong E}{[(\overline{e}, w)]}{[w].}{}{j}
\]

Let $\overline{\phi}$ denote $\id_V \otimes_k \phi$. We have that $\overline{\phi} \circ (\id_V \otimes_k \theta_{E}) = \pm \overline{\phi}$ by hypothesis on $\phi$. Observe also that we pick $\overline{\Psi}$ in $\Hom_X(W_0 \otimes W_0, V \otimes_k \Oo_X)^\pm$, hence $\overline{\Psi} \circ \theta_{W_0} = \pm \overline{\Psi}$. Therefore, 
\[
\overline{\phi} + \overline{\Psi} : (V \otimes E \otimes E) \oplus (W_0 \otimes W_0) \otimes (E \oplus W_0) \to V \otimes_k \Oo_X 
\]
satisfies 
\begin{equation} \label{eq phi + Psi is pm}
(\overline{\phi} + \overline{\Psi}) \circ ((\id_V \otimes_k \theta_{E}) \oplus \theta_{W_0}) = \pm (\overline{\phi} + \overline{\Psi}).
\end{equation}
Recalling that $(\overline{\eta}, \overline{\Psi}) \in V \otimes_k \ker(\delta_0)$, we have $\overline{\phi} \circ ( \pi \otimes \overline{\eta}) + \overline{\Psi} \circ (\id_{W_0} \otimes \partial_{-1}) = 0$, we observe that
\begin{equation} \label{eq descent condition 1 for Phi}
\left . \left (\overline{\phi} + \overline{\Psi} \right) \right|_{(E \oplus W_0) \otimes (\overline{\eta} \oplus \partial_{-1})(W_{-1})} = 0
\end{equation}
As a direct consequence of \eqref{eq phi + Psi is pm} and \eqref{eq descent condition 1 for Phi}, one has 
\begin{equation} \label{eq descent condition 2 for Phi}
\left . \left (\overline{\phi} + \overline{\Psi} \right) \right|_{(\overline{\eta} \oplus \partial_{-1})(W_{-1}) \otimes (E \oplus W_0)} = 0
\end{equation}
Then, $\overline{\phi} + \overline{\Psi}$ defines $\Phi : F \otimes F \to \Oo_X$. Since $\overline{\phi} + \overline{\Psi}$ satisfies \eqref{eq phi + Psi is pm}, it follows that $\Phi$ satisfies \eqref{eq Phi is pm}. Obviously, $\overline{\phi} + \overline{\Psi}$ restricted to $E \oplus 0$ coincides with $\overline{\phi}$, hence $\Phi$ satisfies \eqref{eq Phi restricts to phi} as well.

We now study the action of $\image (\delta_{-1})$ on $\ker (\delta_0)$. For any $(\overline{\eta}, \overline{\Psi}) \in \Hom_X(W_{-1}, E) \oplus \Hom_X(W_0 \otimes W_0, \Oo_X)$ and any $\lambda \in \Hom_X(W_0,E)$, set 
\begin{equation} \label{eq relation between eta Theta and eta' Theta'}
(\overline{\eta}', \overline{\Psi}') = (\overline{\eta}, \overline{\Psi}) + \delta_{-1}(\lambda) = (\overline{\eta} -\lambda \circ \partial_{-1} \, , \,  \overline{\Psi} + \overline{\phi} \circ (\pi \otimes \lambda) + \overline{\phi} \circ (\lambda \otimes \pi) ),
\end{equation}
Let $F'$ be $E \oplus W_0/((-\overline{\eta}') \oplus \partial_{-1})(W_{-1})$. Consider the isomorphism
\begin{equation} \label{eq wt beta}
\begin{pmatrix} \id_E & \lambda \\ 0 & \id_{W_0} \end{pmatrix} : E \oplus W_0 \stackrel{\cong}{\longrightarrow} E \oplus W_0.
\end{equation}
Since the image under \eqref{eq wt beta} of $ (\overline{\eta} \oplus \partial_{-1})(W_{-1})$ is precisely $(\overline{\eta}' \oplus \partial_{-1})(W_{-1})$, this descends to an isomorphism
\[
\lambda_1 : F \stackrel{\cong}{\longrightarrow} F'.
\]
Let $\Phi' \in \Hom_X(F' \otimes F', \Oo_X)^\pm$ defined by $\overline{\phi} + \overline{\Psi}'$. We can check that 
\begin{equation} \label{eq Phi and Phi'}
\Phi = \Phi' \circ (\lambda_1 \otimes \lambda_1),
\end{equation}
so $(\overline{\eta},\overline{\Psi})$ and $(\overline{\eta}', \overline{\Psi}')$ define isomorphic extensions, with this isomorphism relating the corresponding quadratic form.

Conversely, suppose we are given an extension of the form \eqref{eq extension HH^1} and $\Phi : F \otimes F \to \Oo_X$ satisfying \eqref{eq Phi is pm}. Picking a locally free resolution $W^\bullet \stackrel{\pi}{\longrightarrow} E \to 0$, the extension \eqref{eq extension HH^1} determines $\overline{\eta} \in \Hom_X(W_{-1}, E)$ such that $\overline{\eta}(\partial_{-2}(W_{-2})) = 0$. Taking the pull-back of $\Phi$ under $E \oplus W_0 \to F$ and restricting to $W_0$, we obtain $\overline{\Psi} \in \Hom_X(W_0 \otimes W_0, \Oo_X)^\pm$. Since it comes from $\Phi$ defined over $F$, it follows that $\overline{\Psi}$ satisfies \eqref{eq descent condition 1 for Phi} and \eqref{eq descent condition 2 for Phi}. Therefore $(\overline{\eta}, \overline{\Psi})$ lies in $\ker (\delta_0)$. Suppose further that we are given two isomorphic extensions
\[
\xymatrix{
0 \ar[r] & E \ar@{=}[d] \ar[r] & F \ar[d]^{\lambda_1} \ar[r] & E \ar@{=}[d] \ar[r] & 0
\\
0 \ar[r] & E \ar[r] & F' \ar[r] & E \ar[r] & 0,
}
\]
and $\Phi \in \Hom_X(F \otimes F, \Oo_X)^\pm$ and $\Phi' \in \Hom_X(F' \otimes F', \Oo_X)^\pm$ satisfying \eqref{eq Phi and Phi'}. Let $(\overline{\eta}, \overline{\Psi})$ be the element of $\ker(\delta_0)$ associated to the first extension equipped with $\Phi$, and let $(\overline{\eta}', \overline{\Psi}') \in \ker(\delta_0)$ be the pair associated to the second extension and $\Phi'$. Since $\lambda_1$ defines an isomorphism of extensions, it then comes from some isomorphism $E \oplus W_0 \to E \oplus W_0$ of the form \eqref{eq wt beta} for some $\lambda \in \Hom_X(W_0, W_0)$. It then follows that, $\lambda$ is such that \eqref{eq relation between eta Theta and eta' Theta'} holds, so both $(\overline{\eta}, \overline{\Psi})$ and $(\overline{\eta}', \overline{\Psi}')$ are related by the action of $\image (\delta_{-1})$.
\end{proof}

\begin{proposition} \label{pr description HH^2}
Let $E$ be a coherent sheaf over $X$ smooth and projective and consider $\phi : E \otimes E \to \Oo_X$ such that $\phi \circ \theta_E = \pm \phi$. Consider a finite dimensional $k$-vector space $V$. Then, $V \otimes_k \HH^2\left (\Delta^{\pm,\bullet}_{(E,\phi)}\right )$ is the space classifying equivalence classes of $2$-extensions 
\begin{equation} \label{eq extension HH^2}
0 \longrightarrow V \otimes_k E \stackrel{i}{\longrightarrow} F \stackrel{f}{\longrightarrow} G \stackrel{j}{\longrightarrow} E \longrightarrow 0,
\end{equation}
together with a class 
\[
[\mu] \in \quotient{\Hom_X(G \otimes  F, V \otimes_k \Oo_X)}{\left (\Hom_X(G \otimes G, V \otimes_k \Oo_X)^\pm \circ (\id_{G} \otimes f) \right )}
\]
whose elements satisfy
\begin{equation} \label{eq mu and phi}
\mu \circ (1_{G} \otimes i) = \id_V \otimes_k \phi(j \otimes 1_E),
\end{equation}
and
\begin{equation} \label{eq mu and nu}
\mu \circ (f \otimes \id_{F}) = \pm \mu \circ (f \otimes \id_{F}) \circ \theta_{F}.
\end{equation}

The zero element in $V \otimes_k \HH^2(\Delta^{\pm, \bullet}_{(E,\phi)})$ corresponds to a $2$-extension that splits
\begin{equation} \label{eq split ses}
0 \longrightarrow V \otimes_k E \stackrel{i}{\longrightarrow} F = (V \otimes_k E) \oplus \ker j \stackrel{f}{\longrightarrow} G \stackrel{j}{\longrightarrow} E \longrightarrow 0,
\end{equation}
and $[\mu]$ such that
\begin{equation} \label{eq condition for mu to lift}
[\mu] \circ (\id_{G} \otimes q) = 0  
\end{equation}
in $\Ext^1(E\otimes E, V \otimes_k \Oo_X)^\pm$, where $q$ denotes the projection $F \to \ker j$. 
\end{proposition}

\begin{proof}
We study $V \otimes_k \HH^1\left (C^\bullet \left ( \Delta^\pm_{(E,\phi)} \right ) \right )$. Using the locally free resolution $W^\bullet \stackrel{\pi}{\longrightarrow} E \to 0$, the mapping cone of $\Delta^{\pm}_{(E,\phi)}$ in degrees $1$ and $2$ is given by 
\[
\begin{pmatrix} \Hom_X(W_{-2}, E) \\ \Hom_X(W_0 \otimes W_{-1}, \Oo_X) \end{pmatrix} \stackrel{\delta_1}{\longrightarrow} 
\begin{pmatrix} \Hom_X(W_{-3}, E) \\ \Hom_X(W_{-1} \otimes  W_{-1}, \Oo_X)^\mp \\ \Hom_X(W_0 \otimes W_{-2}, \Oo_X) \end{pmatrix},
\]
where
\[
\delta_1 = \begin{pmatrix}
-(\bullet) \circ \partial_{-3} & 0
\\
0 &  (\bullet) \circ (\partial_{-1} \otimes \id_{W_{-1}}) \circ (\id_{W_{-1}} \mp \theta_{W_{-1}} ) 
\\
\phi \circ (\pi \otimes (\bullet)) & - (\bullet) \circ (\id_{W_0} \otimes \partial_{-2})
\end{pmatrix}.
\]
Then, 
\[
\ker(\delta_1) = \left \{ \begin{matrix} (\chi, \Xi) \in \Hom_X(W_{-2}, E) \oplus \Hom_X(W_0 \otimes W_{-1}, \Oo_X)  \\   \chi (\partial_{-3}(W_{-3})) = 0 \\ \Xi \circ (\partial_{-1} \otimes \id_{W_{-1}}) = \pm \, \Xi \circ (\partial_{-1} \otimes \id_{W_{-1}}) \circ \theta_{W_{-1}} \\ \phi \circ (\pi \otimes \chi) - \Xi \circ (\id_{W_0} \otimes \partial_{-2}) = 0 \end{matrix} \right \}.
\]
Using $\left (\overline{\chi}, \overline{\Xi} \right ) \in V \otimes_k \ker(\delta_1)$, where $\overline{\chi} \in \Hom_X(W_{-1}, V \otimes_k E)$ and $\overline{\Xi} \in \Hom_X(W_0 \otimes W_{-1}, V \otimes_k \Oo_X)$ consider the projection $j: G \to E$ to be $\pi : W_0 \to E$ and let us construct 
\begin{equation} \label{eq k-definition of wt E}
F := (V \otimes_k E) \oplus W_{-1} / \left (\overline{\chi} \oplus \partial_{-2} \right )(W_{-2})
\end{equation}
and consider the injection
\[
\morph{E}{F = (V \otimes_k E) \oplus W_{-1} / \left (\overline{\chi} \oplus \partial_{-2} \right )(W_{-2})}{\overline{e}}{[(\overline{e}, 0)]}{}{i}
\]
and the morphism
\[
\morph{F = (V \otimes_k E) \oplus W_{-1} / \left (\overline{\chi} \oplus \partial_{-2} \right )(W_{-2})}{G = W_{0}}{[(\overline{e}, w)]}{\partial_{-1}(w),}{}{f}
\]
which is well defined since $\partial_{-1} \circ \partial_{-2} = 0$. Note that we have obtained a $2$-extension of the form of \eqref{eq extension HH^2}.

Consider now 
\[
\id_V \otimes_k \phi \circ (\pi \otimes \id_E) - \overline{\Xi} : W_0 \otimes ((V \otimes_k E) \oplus W_{-1}) \longrightarrow V \otimes_k \Oo_X.
\]
Since $(\overline{\chi}, \overline{\Xi}) \in V \otimes_k \ker(\delta_1)$ satisfy $\id_V \otimes_k \phi \circ (\pi \otimes \overline{\chi}) - \overline{\Xi} \circ (\id_{W_0} \otimes \partial_{-2}) = 0$, it follows that $\id_V \otimes_k \phi \circ (\pi \otimes \id_E) - \overline{\Xi}$ vanishes at $W_0 \otimes ( (\overline{\chi} \oplus \partial_{-2})(W_{-2}))$, hence it descends to
\[
\mu : G \otimes F \to V \otimes_k \Oo_X,
\]
where we recall \eqref{eq k-definition of wt E} and that $G = W_0$. Since $\id_V \otimes_k \phi \circ (\pi \otimes \id_E) - \overline{\Xi}$ restricted to $G\otimes (V \otimes_k E \oplus 0)$ amounts to $\id_V \otimes_k \phi \circ (\pi \otimes \id_E)$, \eqref{eq mu and phi} follows naturally. Note also that 
\[
(\id_V \otimes_k \phi \circ (\pi \otimes \id_E) - \overline{\Xi}) \circ (\partial_{-1} \otimes \id_{(V \otimes_k E) \oplus W_{-1}}) = \overline{\Xi} \circ (\partial_{-1} \otimes \id_{W_{-1}}).
\]
Then \eqref{eq mu and nu} follows from the identity $\overline{\Xi} \circ (\partial_{-1} \otimes \id_{W_{-1}}) = \pm \, \overline{\Xi} \circ (\partial_{-1} \otimes \id_{W_{-1}}) \circ \theta_{W_{-1}}$ that any $(\overline{\chi}, \overline{\Xi}) \in V \otimes_k \ker(\delta_1)$ satisfies.

For any $\overline{\eta} \in V \otimes_k \Hom_X(W_{-2}, E)$, we set
\[
(\overline{\chi}', \overline{\Xi}') = (\overline{\chi}, \overline{\Xi}) + (\id_V \otimes_k \delta_0) \cdot (\overline{\eta}, 0) = \left ( \overline{\chi} - \overline{\eta} \circ \partial_{-2}, \overline{\Xi} + \id_V \otimes_k \phi \circ ( \pi \otimes \eta) \right ),
\]
and define
\[
F' = (V \otimes_k E) \oplus W_{-1} / (\overline{\chi}' \oplus \partial_{-2})(W_{-2}). 
\]
Observe that the isomorphism
\begin{equation} \label{eq wt alpha}
\begin{pmatrix} \id_{V \otimes_k E} & -\overline{\eta} \\ 0 & \id_{W_0} \end{pmatrix} : (V \otimes_k E) \oplus W_{-1} \stackrel{\cong}{\longrightarrow} (V \otimes_k E) \oplus W_{-1}.
\end{equation}
sends $(\overline{\chi} \oplus \partial_{-2})(W_{-2})$ to $(\overline{\chi}' \oplus \partial_{-2})(W_{-2})$, hence \eqref{eq wt alpha} provides an isomorphism 
\[
\eta_F : F \stackrel{\cong}{\longrightarrow} F'.
\]
One obtains a commutative diagram
\[
\xymatrix{
0 \ar[r] & V \otimes_k E \ar[r] \ar@{=}[d] & F \ar[r] \ar[d]^{\eta_F}_{\cong} & G = W_0 \ar[r] \ar@{=}[d] & E \ar[r] \ar@{=}[d] & 0 
\\
0 \ar[r] & V \otimes_k E \ar[r] & F' \ar[r] & G' = W_0 \ar[r] & E \ar[r] & 0,
}
\]
so both $\overline{\chi}$ and $\overline{\chi}'$ define the same class of extensions. Note also that $\id_V \otimes_k \phi \circ (\pi \otimes \id_E) - \overline{\Xi}'$ vanishes in $(\overline{\chi}' \oplus \delta_{-2})(W_{-2})$, defining $\mu': G' \otimes F' \to V \otimes_k \Oo_X$ satisfying \eqref{eq mu and phi} and \eqref{eq mu and nu}. Note also that
\[
\id_V \otimes_k \phi \circ (\pi \otimes \id_E) - \overline{\Xi}' = \id_V \otimes_k (\phi \circ (\pi \otimes \id_E) - \overline{\Xi}) \circ  \begin{pmatrix} \id_{V \otimes_k E} & -\overline{\eta} \\ 0 & \id_{W_0} \end{pmatrix},
\]
so one gets
\[
\mu' \circ (\id_{G} \otimes \eta_1) = \mu.
\]
Therefore, the $2$-extension and the morphism that we obtain from $(\overline{\chi}, \overline{\Xi})$ are equivalent to the $2$-extension and the morphism that we obtain from $(\overline{\chi}', \overline{\Xi}')$. 

Take now $\overline{\Psi} \in V \otimes_k \Hom_X(W_0 \otimes W_0, \Oo_X)^\pm$ and define 
\[
\left ( \overline{\chi}, \overline{\Xi}'' \right ) = \left ( \overline{\chi}, \overline{\Xi} \right ) + (\id_V \otimes_k \delta_0) \cdot \left ( 0, \overline{\Psi} \right ) = \left (\overline{\chi}, \overline{\Xi} + \overline{\Psi} \circ (\id_{W_0} \otimes \partial_{-1}) \right ).
\]
Since $\overline{\chi}$ does not change, we get $F$ and $G$ as before. We observe that $\id_V \otimes_k \phi \circ (\pi \otimes \id_E) - \overline{\Xi}''$ descends to $\mu'' = \mu + \overline{\Psi} \circ (\id_{G} \otimes f)$, so $[\mu''] = [\mu]$. Observe that $\mu''$ also satisfies \eqref{eq mu and phi} and \eqref{eq mu and nu}.

Conversely, choose representants \eqref{eq extension HH^2} and $\mu: G \otimes F \to V \otimes_k \Oo_X$ of a given $2$-extension class of $E$ by $E$, and of a certain class in $\Hom_X(G \otimes  F, V \otimes_k \Oo_X)/\Hom_X(G \otimes G, V \otimes_k \Oo_X)^\pm \circ (\id_{W_0} \otimes f)$  satisfying \eqref{eq mu and phi} and \eqref{eq mu and nu}. Using the universal property of projective modules (recall that locally free sheaves are projective) one can always complete to a commutative diagram
\begin{equation} \label{eq construction of chi}
\xymatrix{
0 \ar[r] & W_{-2} / \partial_{-3}(W_{-3}) \ar[r] \ar[d]^{\overline{\chi}} & W_{-1} \ar[r] \ar[d]^{\chi_F} \ar[r] & W_0 \ar[r] \ar[d]^{\chi_G} & E \ar[r] \ar@{=}[d] & 0
\\
0 \ar[r] & V \otimes_k E \ar[r] & F \ar[r] & G \ar[r] & E \ar[r] & 0.
}
\end{equation}
This defines a morphism $\overline{\chi} : W_{-2} \to V \otimes_k E$ with $\overline{\chi}( \partial_{-3}(W_{-3})) = 0$. Let us denote the composition of $\overline{\Xi} := \mu \circ (\chi_G, \chi_F)$. Note that the commutativity of \eqref{eq construction of chi} together with \eqref{eq mu and phi} and \eqref{eq mu and nu} imply, respectively, that 
\[
\id_V \otimes_k \phi \circ (\pi \otimes \overline{\chi}) - \overline{\Xi} \circ (\id_{W_0} \otimes \partial_{-2}) = 0
\]
and
\[
\overline{\Xi} \circ (\partial_{-1} \otimes \id_{W_{-1}}) = \pm \, \overline{\Xi} \circ (\partial_{-1} \otimes \id_{W_{-1}}) \circ \theta_{W_{-1}}.
\]
This completes the proof of the first statement.

To describe the $0$ element in $V \otimes_k \HH^1\left (C^\bullet \left ( \Delta^{\pm}_{(E,\phi)} \right ) \right )$, we first note that any $\left (0,\overline{\Xi} \right ) \in V \otimes_k \ker(\delta_1)$ gives
\[
F = (V \otimes_k E) \oplus W_{-1} / (0 \oplus \delta_{-2})(W_{-2}) = (V \otimes_k E) \oplus \ker \pi,
\]
with $\overline{\Xi}$ satisfying 
\[
\overline{\Xi} \circ (\partial_{-1} \otimes \id_{W_{-1}}) = \pm \, \overline{\Xi} \circ (\partial_{-1} \otimes \id_{W_{-1}}) \circ \theta_{W_{-1}} 
\]
and
\[ 
\overline{\Xi} \circ (\id_{W_0} \otimes \partial_{-2}) = 0.
\]
Therefore, $\left (0,\overline{\Xi} \right ) \in \ker(\delta_1)$ determines a short exact sequence of the form \eqref{eq split ses} together with an element $\left [\, \overline{\Xi} \, \right ] \in \Ext^1_X(E \otimes E, V \otimes_k \Oo_X)^{\pm}$. If $\mu$ is the descent of $\id_V \otimes \phi \circ (\pi \otimes \id_E) - \overline{\Xi} : W_0 \otimes (V \otimes_k E \oplus W_{-1}) \to V \otimes_k \Oo_X$ to a morphism in $\Hom_X(W_0 \otimes \ker(\pi), V \otimes_k \Oo_X)^{\pm}$, observe that $\mu \circ (\id_{G} \otimes q)$ coincides with $\overline{\Xi}$. This concludes the proof.
\end{proof}

\section{Deformation of orthogonal and symplectic sheaves}
\label{sc deformation}

Given an orthogonal sheaf $(E, \phi)$ over the projective scheme $X$, we define its deformation functor
\[
\Def^{\, +}_{(E,\phi)} : \Art \longrightarrow \Sets
\]
by associating to any $A \in \Art$ the set of isomorphism classes of triples $(\Ee, \Phi, \gamma)$, where $(\Ee,\Phi)$ is a family of orthogonal sheaves on $X_A$, (so $\Ee$ is a torsion-free coherent sheaf on $X_A$ flat over $\Spec(A)$), and $\gamma: (\Ee, \Phi)|_X \to (E,\phi)$ is an isomorphism of orthogonal (resp. symplectic) sheaves. Two triples $(\Ee, \Phi, \gamma)$ and $(\Ee', \Phi', \gamma')$ are isomorphic if there exists an isomorphism $f : (\Ee,\Phi) \to (\Ee',\Phi')$ such that $\gamma' \circ (f|_X) = \gamma$. As in the case of the deformation functor of coherent sheaves, functoriality of $\Def^{\, +}_{(E,\phi)}$ follows from applying pull-backs under the morphisms $p_{a}:X_{A'} \to X_A$ for any $a : A \to A'$.

Analogously, associated to any symplectic sheaf $(E, \phi)$ over $X$, its deformation functor
\[
\Def^{\, -}_{(E,\phi)} : \Art \longrightarrow \Sets
\]
is constructed by associating to every $A \in \Art$ the set of isomorphism classes of triples $(\Ee, \Phi, \gamma)$, where $(\Ee,\Phi)$ is a family of symplectic sheaves on $X_A$, and $\gamma: (\Ee, \Phi)|_X \to (E,\phi)$ is an isomorphism of symplectic sheaves. The notion of isomorphism of triples is analogous to the case of orthogonal sheaves. As before, functoriality under pull-backs holds in this case as well.

We see that the complex $\Delta^{\pm,\bullet}_{(E,\phi)}$ governs the deformation theory of $\Def^{\, \pm}_{(E,\phi)}$.

\begin{theorem}
Let $(E,\phi)$ be an orthogonal ({\it resp.} symplectic) sheaf over the smooth projective scheme $X$. Then the deformation functor $\Def^{\, +}_{(E,\phi)}$ ({\it resp.} $\Def^{\, -}_{(E,\phi)}$) admits a miniversal pro-family and the associated space of first-order deformations is
\begin{equation} \label{eq description of 1st def^+}
\Def^{\, \pm}_{(E,\phi)}(k[\epsilon]/(\epsilon^2)) \cong \HH^1 \left ( \Delta^{\pm,\bullet}_{(E,\phi)} \right ).
\end{equation}
If, further, $(E,\phi)$ is simple, $\Def^{\, \pm}_{(E,\phi)}$ is pro-representable.
%
\end{theorem}

\begin{proof}
Given a torsion-free sheaf $E$ with $\phi : E \otimes E \to \Oo_X$ satisfying $\phi = \pm \phi \circ \theta_E$, we have to check whether $\Def^{\, \pm}_{(E,\phi)}$ satisfies the Schlessinger conditions \textbf{S1}, \textbf{S2} and \textbf{S3}.   
Let us denote by $b : B \times_A C \to B$ the projection to the first factor and by $c: B \times_A C \to C$ the projection to the second. Consider the associated morphisms $p_b : X_B \to X_{ B \times_A C}$ and $p_c : X_{ B \times_A C} \to X_C$. Condition \textbf{S1} holds if for any homomorphism $B \to A$ and any small extension $0 \to H \to C \to A \to 0$ in $\Art$, the morphism induced by taking pull-backs under $p_b$ and $p_c$,
\[
\Def^{\, \pm}_{(E,\phi)}(B \times_A C) \longrightarrow \Def^{\, \pm}_{(E,\phi)}(B) \times_{\Def^{\, \pm}_{(E,\phi)}(A)} \Def^{\, \pm}_{(E,\phi)}(C),
\]
is surjective. Consider $(\Ee_B, \Phi_B,\gamma_B) \in \Def^{\, \pm}_{(E,\phi)}(B)$ and $(\Ee_C,\Phi_C, \gamma_C) \in \Def^{\, \pm}_{(E,\phi)}(C)$ for which there exists an isomorphism $f : \Ee_B|_{X_A} \to \Ee_C|_{X_A}$ satisfying that $\gamma_B = \gamma_C \circ f|_{X}$ and $\Phi_B|_{X_A} = \Phi_C|_{X_A} \circ (f \otimes f)$. Thanks to deformation theory of sheaves we know that $\Def_E$ verifies $\S1$, so 
\[
\Def_E(B \times_A C) \longrightarrow \Def_E(B) \times_{\Def_E(A)} \Def_E(C),
\]
is surjective and there exists $(\Ee', \gamma) \in \Def_E(B \times_A C)$ such that $g_B : (\Ee', \gamma)|_{X_B} \stackrel{\cong}{\to} (\Ee_B, \gamma_B)$ and $g_C: (\Ee', \gamma)|_{X_C} \stackrel{\cong}{\to} (\Ee_C, \gamma_C)$ satisfy that $g_C|_{X_A} = f \circ g_B|_{X_A}$. It remains to construct a quadratic form on $\Ee'$ compatible with $\Phi_B$ and $\Phi_C$ under pull-backs.

Write $\Ee'_B$, $\Ee'_C$ and $\Ee'_A$ for $\Ee'|_{X_B}$, $\Ee'|_{X_C}$ and $\Ee'|_{X_A}$ respectively. One trivially has that 
\begin{equation} \label{eq decomposition of Ff}
\Ee' = \Ee'_B \times_{\Ee'_A} \Ee'_C, 
\end{equation}
so
\[
\Ee' \otimes \Ee' = (\Ee'_B \otimes \Ee'_B)  \times_{(\Ee'_A \otimes \Ee'_A)} (\Ee'_C \otimes \Ee'_C). 
\]
Since $g_C|_{X_A} = f \circ g_B|_{X_A}$ and $\Phi_B|_{X_A} = \Phi_C|_{X_A} \circ (f \otimes f)$, it follows that 
\[
(g_B \otimes g_B)^*\Phi_B |_{X_A} = (g_C \otimes g_C)^*\Phi_C |_{X_A}.
\]
Then, they define 
\[
\Phi : \Ee' \otimes \Ee' \longrightarrow \Oo_{X_{B \times_A C}}, 
\]
which naturally satisfies $p_b^* \Phi = \Phi_B$ and $p_c^*\Phi = \Phi_C$. By hypothesis, one has that $p_b^*\Phi_B = \pm p_b^*\Phi_B \circ \theta_{\Ee'_B}$ and similarly for $C$ and $A$. It then follows that 
\[
\Phi = \pm \Phi \circ \theta_{\Ee'}.
\]
Recall that $U_{\Ee'}$ the open subset of $X_{(B \times_A C)}$ where is $\Ee'$ is locally free. Thanks to \eqref{eq decomposition of Ff}, one has that 
\[
U_{\Ee'} = U_{\Ee'_B} \times_{U_{\Ee'_A}} U_{\Ee'_C}.
\]
By hypothesis, $(g_B \times g_B)^*\Phi_B$ and $(g_C \times g_C)^*\Phi_C$ and are non-degenerate over $U_{\Ee_B} = U_{\Ee'_B}$ and $U_{\Ee_C} = U_{\Ee'_C}$. Since $\Phi$ is constructed by gluing $(g_B \times g_B)^*\Phi_B$ and $(g_C \times g_C)^*\Phi_C$, it then follows that $\Phi$ is non-degenerate over $U_{\Ee'}$. This proves that $\Def^{\, \pm}_{(E,\phi)}$ satisfies \textbf{S1}.

Since $\Def_E$ satisfies the condition \textbf{S2}, to prove that $\Def^{\, \pm}_{(E,\phi)}$ also full-fills this condition it is enough to see that, when $A = k$, given a quadratic form $\Phi'$ on $\Ee'$ compatible with $\gamma$ and such that $\Phi'|_{\Ee'_B \otimes \Ee'_B} = (h_B \otimes h_B)^*\Phi|_{\Ee'_B \otimes \Ee'_B}$ for an automorphism $h_B$ of $(\Ee'_B, \gamma|_{\Ee'_B})$ and $\Phi'|_{\Ee'_C \otimes \Ee'_C} = (h_C \otimes h_C)^*\Phi|_{\Ee'_C \otimes \Ee'_C}$ under an automorphism $h_C$ of $(\Ee'_C, \gamma|_{\Ee'_C})$, one can construct an automorphism $h$ of $(\Ee', \gamma)$ sending $\Phi'$ to $\Phi$. Since $h_B|_{X} = \gamma|_{X} = h_C|_{X}$, it follows from \eqref{eq decomposition of Ff} that one can construct $h : \Ee' \to \Ee'$ by gluing $h_B$ and $h_C$ along $X$. It is straight-forward that $h$ satisfies the required condition so $\Def^{\, \pm}_{(E,\phi)}$ satisfies \textbf{S2}. 
 
We address \textbf{S3} now. First, note that one can endow the space of first-order deformations $\Def^{\, \pm}_{(E,\phi)}(k\langle k \rangle)$ with a $k$-vector space structure. Using the inverse of the bijective map of \textbf{S2} when $B = k\langle k \rangle$, and the morphism $\langle + \rangle : k\langle k \oplus k \rangle \to k\langle k\rangle$ induced by the sum of the elements in the maximal ideal, one can define the sum within the space of first-order deformations,
\[
\xymatrix{
\Def^{\, \pm}_{(E,\phi)}(k\langle k \rangle) \times \Def^{\, \pm}_{(E,\phi)}(k\langle k \rangle) \ar[r]^{\qquad \, \, \, 1:1} & \Def^{\, \pm}_{(E,\phi)}(k\langle k \oplus k \rangle) \ar[rr]^{\quad \Def^{\, \pm}_{(E,\phi)}(\langle + \rangle)} & & \Def^{\, \pm}_{(E,\phi)}(k\langle k \rangle).
}
\]
One can easily check that $\Def^{\, \pm}_{(E,\phi)}(k\langle k \rangle)$ equipped with this sum satisfies the axioms of a $k$-vector space. 

Recall now that $k\langle k \rangle = k[\epsilon]/(\epsilon^2)$. Suppose that $M$ is a $\Oo_{X} \times_k k[\epsilon]/(\epsilon^2)$-module. Since $\Oo_{X} \times_k k[\epsilon]/(\epsilon^2) = \Oo_X + \epsilon \Oo_X$ with $\epsilon^2 = 0$, we see that a $\Oo_{X} \times_k k[\epsilon]/(\epsilon^2)$-module structure is determined by a $\Oo_X$-module structure and the action of $\epsilon$ on it, $\epsilon k \times_k M \to k \times_k M$. Thanks to Proposition \ref{pr description HH^1}, one can then consider the map
\begin{equation} \label{eq tangent to Def_E phi}
\HH^1(\Delta^{\pm}_{(E,\phi)}) \longrightarrow \Def^{\, \pm}_{(E,\phi)}(k[\epsilon]/(\epsilon^2))
\end{equation}
that sends the short exact sequence of $\Oo_X$-modules $0 \to E \stackrel{i}{\to} F \stackrel{j}{\to} E \to 0$ and $\Phi_1 : F \otimes F \to \Oo_X$ to the triple $(\Ee, \Phi, \gamma)$ where $\Ee$ is $F$ endowed with a $\Oo_{X} \times_k k[\epsilon]/(\epsilon^2)$-module structure determined by the usual $\Oo_X$-module structure on $F$ and the action of $\epsilon$ on $\Ee$ defined by the composition $j \circ i : E \to E$. The isomorphism $\gamma : \Ee|_X \to E$ is determined by the projection $j : F \to E$ and $\Phi$ is naturally determined by $\Phi_1$. 

Conversely, given the isomorphism class of $(\Ee, \Phi, \gamma)$ in $\Def^{\, \pm}_{(E,\phi)}(k[\epsilon]/(\epsilon^2))$, we construct an exact sequence $0 \to E \stackrel{i}{\to} F \stackrel{j}{\to} E \to 0$. Tensorize $\Ee$ with $0 \to k \to k[\epsilon]/(\epsilon^2) \to k \to 0$ to obtain
\[
0 \longrightarrow k \otimes_{k[\epsilon]/(\epsilon^2)} \Ee  \cong E \stackrel{i'}{\longrightarrow} \Ee \stackrel{j'}{\longrightarrow} k \otimes_{k[\epsilon]/(\epsilon^2)} \Ee \cong E   \longrightarrow 0.
\]
Then consider the push-forward under the natural projection $\pi : X \times_k \Spec (k[\epsilon]/(\epsilon^2)) \to X$ gives a short exact sequence given by $F = \pi_*\Ee$, the projection $j : F \to E$ is determined by the composition $\gamma \circ \pi_*j'$ and $\Phi_1 = (\pi \otimes \pi)^*\Phi$. This provides an inverse for \eqref{eq tangent to Def_E phi} so \textbf{S3} is satisfied and \eqref{eq description of 1st def^+} holds.

Finally, we address \textbf{S4} for $(E,\phi)$ simple. Observe that it follows naturally from \textbf{S1} and Lemma \ref{lm surjection of automorphisms groups}. 
\end{proof}

The following result provides a characterization of simple orthogonal and symplectic sheaves on $X_A$ in terms of their restriction to the the closed subset $X \subset X_A$.

\begin{lemma} \label{lm surjection of automorphisms groups}
For any $A \in \Art$ and any orthogonal ({\em resp.} symplectic) sheaf $(\Ee_A,\Phi_A)$ over $X_A$, we have that $(\Ee_A,\Phi_A)$ is simple if and only if $(E,\phi) := (\Ee_A|_X,\Phi_A|_X)$ is simple.
\end{lemma}

\begin{proof}
The $A$-module $H^0(X_A, \End(\Ee_A))$ is finitely generated. Note that $H^0(X_A, \End(\Ee_A)) \otimes_A k = H^0(X,\End(E))$ and $\id_{\Ee_A} \otimes_A 1 = \id_E$. 

By Nakayama's lemma, if $\{ \id_{\Ee_A}, b_2, \dots, b_n\} \subset H^0(X_A, \End(\Ee_A))$ are such that $\{ \id_E, (b_2 \otimes_A 1), \dots, (b_n  \otimes_A 1)\}$ is a basis of $H^0(X,\End(E))$, then $\{ \id_{\Ee_A}, b_2, \dots, b_n\}$ generate $H^0(X_A, \End(\Ee_A))$. Then, there is no $b' \in H^0(X_A, \End(\Ee_A))$ such that $b' \neq \id_{\Ee_A}$ with restriction $b' \otimes_A 1 = \id_E$. Otherwise, as $b' = a_1 \otimes_A \id_{\Ee_A} + a_2 \otimes_A b_2 + \dots + a_n \otimes_A b_n$, we would obtain a contradiction with the linear independence of $\{ \id_E, (b_2 \otimes_A 1), \dots, (b_n  \otimes_A 1)\}$.

Then, the only element in $H^0(X_A, \End(\Ee_A))$ that restrict to $\id_E$ is $\id_{\Ee_A}$ itself. If $\Aut(E,\phi) = \{ \id_E, -\id_E \}$, it then follows that $\Aut(\Ee_A,\Phi_A)$ is $\{\id_{\Ee_A}, -\id_{\Ee_A}\}$.  
\end{proof}

\section{Obstruction theory for orthogonal and symplectic sheaves}
\label{sc obstruction}

In this section we will see that the second cohomology space of the deformation complex defined in Section \ref{sc def complex} provides an obstruction theory for the deformation functors of orthogonal and symplectic sheaves. We begin by the construction of the morphism \eqref{eq obs morphism} in this case.

Let $(E,\phi)$ be an orthogonal (resp. symplectic) sheaf over $X$ projective and let $0 \to H \to B \stackrel{\tau}{\to} A \to 0$ be a small extension of Artin algebras with residue field $k$. We want to construct a morphism from $\Def^{\, \pm}_{(E,\phi)}(A)$ to $H \otimes_k \HH^2 \left ( \Delta^{\pm,\bullet}_{(E,\phi)}\right )$. Then, after Proposition \ref{pr description HH^2}, given $(\Ee_A, \Phi_A, \gamma_A) \in \Def^{\, \pm}_{(E,\phi)}(A)$, we should construct a $2$-extension of the form \eqref{eq extension HH^2} equipped with a class of morphisms $[\mu] \in \Hom_X(G \otimes  F, \Oo_X)/\left (\Hom_X(G \otimes G, \Oo_X)^\pm \circ (\id_{G} \otimes f) \right )$ satisfying \eqref{eq mu and phi} and \eqref{eq mu and nu}.

From the small extension $0 \to H \to B \stackrel{\tau}{\to} A \to 0$, one naturally obtains the short exact sequence of $\Oo_{X_B}$-modules
\begin{equation} \label{eq ses structural sheaves}
0 \longrightarrow H \otimes_k \Oo_X \stackrel{\sigma_0}{\longrightarrow} \Oo_{X_B} \stackrel{\rho_0}{\longrightarrow} p_* \Oo_{X_A} \longrightarrow 0,
\end{equation}
where we denote by $p : X_A \hookrightarrow X_B$ the morphism associated to $B \twoheadrightarrow A$.  

Let us consider a locally free resolution $\Ww_A^\bullet \stackrel{\pi_A}{\longrightarrow} \Ee_A \to 0$ such that, for $i>0$,  
\[
\H^i(X_A, \Ww_{A,0}) = 0.
\]
Along with the above locally free resolution, consider a locally free sheaf $\Ww_{B,0}$ satisfying $\Ww_{B,0}|_{X_A} \cong \Ww_{A,0}$ and
\begin{equation} \label{eq cohomology vanishing for W_B} 
\H^i(X_B, \Ww_{B,0}) = 0,
\end{equation}
for $i > 0$. It can be easily verified that such a choice exists. Set $W_i := \Ww_{A,i}|_X$, with differentials $\partial_i := \partial_{A,i}|_X$ and $\pi := \gamma_A \circ \pi_A|_X : W_0 \to E$. Note that $W^\bullet \stackrel{\pi}{\longrightarrow} E \to 0$ is a locally free resolution.


Inspired by the classical approach to obstruction theory of coherent sheaves, we consider the short exact sequences of $\Oo_{X_B}$-modules
\begin{equation} \label{eq diagram to be filled}
\xymatrix{
 & 0 \ar[d] &  & 0 \ar[d] & 
\\
 &  H \otimes_k \partial_{-1}(W_{-1}) \ar[d] &  & p_*\partial_{A,-1}(\Ww_{A,-1}) \ar[d] & 
\\
0 \ar[r] &  H \otimes_k W_0 \ar[d] \ar[r]^\sigma & \Ww_{B,0} \ar[r]^\rho & p_* \Ww_{A,0} \ar[d]^{\pi_A} \ar[r] & 0
\\
 &  H \otimes_k E \ar[d] &  & p_* \Ee_{A} \ar[d] & 
\\
 & 0 & & 0, & 
}
\end{equation}
were $\rho$ is induced by the restriction $X_B \to X_A$ and $\sigma$ is the inclusion of the kernel of $\rho$. Denote
\begin{equation}\label{eq def of Ff}
\Ff := \quotient{\rho^{-1}\left ( p_*\partial_{A,-1}(\Ww_{A,-1}) \right )}{\sigma \left ( H \otimes_k \partial_{-1}(W_{-1}) \right )}.
\end{equation}
We then see that, out of \eqref{eq diagram to be filled}, one can naturally construct an extension 
\begin{equation} \label{eq ses obs th}
0 \longrightarrow H \otimes_k E \stackrel{i_0}{\longrightarrow} \Ff \stackrel{f_0}{\longrightarrow} p_* \partial_{A,-1}(\Ww_{A,-1}) \longrightarrow 0,
\end{equation}
which defines naturally $\zeta \in \Ext^1_{X_B}(p_*\partial_{A,-1}(\Ww_{A,-1}),  H \otimes_k E)$.

Having \eqref{eq ses structural sheaves} in mind, observe that any $\Oo_{X_B}$-module $\Mm$ such that $(H \otimes_k \Oo_X) \otimes \Mm = 0$ is naturally equipped with the inherited $\Oo_{X_B}/(H \otimes_k \Oo_X) \cong p_*\Oo_{X_A}$-module structure. Note that there exists an equivalence of categories sending the $p_*\Oo_{X_A}$-module $\Mm$ to the $\Oo_{X_A}$-module $\Mm^A$. Under this equivalence, $p_* \partial_{A,-1}(\Ww_{A,-1})$ gives naturally the $\Oo_{X_A}$-module $\partial_{A,-1}(\Ww_{A,-1})$. Since $H^2=0$, we see that $H \otimes_k E$ is annihilated by $H \otimes_k \Oo_X$ producing the $\Oo_{X_A}$-module $H \otimes_k \otimes p_{A,*} E$, where $p_{A} : X \hookrightarrow X_A$ is the inclusion of the closed reduced subscheme associated to the structural projection to the residue field, $A \to k$. Hence, after \eqref{eq ses obs th} and right-exactness of tensor product, one has that $(H \otimes_k \Oo_X) \otimes \Ff = 0$ so $\Ff$ gives rise to the $\Oo_{X_A}$-module $\Ff^A$. Therefore, from \eqref{eq ses obs th} we obtain the extension of $\Oo_{X_A}$-modules
\begin{equation} \label{eq ses obs th in X_A}
0 \longrightarrow H \otimes_k p_{A,*} E \stackrel{i_0^A}{\longrightarrow} \Ff^A \stackrel{f_0^A}{\longrightarrow} \partial_{A,-1}(\Ww_{A,-1}) \longrightarrow 0,
\end{equation}
associated to $\zeta^A \in \Ext^1_{X_A}(\partial_{A,-1}(\Ww_{A,-1}), H \otimes_k p_{A,*} E)$.

Composing \eqref{eq ses obs th in X_A} with the projection $\Ww_{A,0} \stackrel{\pi_A}{\to} \Ee_A$, one gets an element $\pi_A^*\zeta^A$ of $\Ext^2_{X_A}(\Ee_A, H \otimes_k p_{A,*} E)$ associated to the $2$-extension
\begin{equation} \label{eq 2-ext obs th in X_A}
0 \longrightarrow  H \otimes_k p_{A,*} E \stackrel{i_0^A}{\longrightarrow} \Ff^A \stackrel{f_0^A}{\longrightarrow} \Ww_{A,0} \stackrel{\pi_A}{\longrightarrow} \Ee_A \longrightarrow 0.
\end{equation}

One can prove that $\Ext^2_{X_A}(\Ee_A, H \otimes_k p_{A,*} E) \cong \Ext^2_X(E,H \otimes_k E)$, so \eqref{eq 2-ext obs th in X_A} is completely determined by its restriction to $X$,
\begin{equation} \label{eq getting 2-extension}
0 \longrightarrow  H \otimes_k E \stackrel{i}{\longrightarrow} F:= \Ff^A |_X \stackrel{f}{\longrightarrow} W_0 \stackrel{\pi}{\longrightarrow} E \longrightarrow 0.
\end{equation}
Note that $i$ and $f$ are given respectively by $i_0^A|_X$, $f_0^A|_X$ and we recall that $\pi = \gamma_A \circ \pi_A|_X$. Setting $G := W_0$ and $j := \pi$, we see that \eqref{eq getting 2-extension} gives a $2$-extension of the form \eqref{eq extension HH^2}.

Take now $\Phi_A :  \Ee_A \otimes \Ee_A \to \Oo_{X_A}$, and consider $(\pi_A \otimes \pi_A)^* \Phi_A \in \Hom_{X_A}(\Ww_{A,0} \otimes \Ww_{A,0}, \Oo_{X_A})^{\pm}$. Recalling \eqref{eq cohomology vanishing for W_B}, one naturally has that $\Ext_{X_B}^1(\Ww_{B,0} \otimes \Ww_{B,0}, \Oo_{X_B}) = 0$, hence the functor $\Hom_{X_B}(\Ww_{B,0} \otimes \Ww_{B,0}, \bullet)^{\pm}$ applied to the short exact sequence \eqref{eq ses structural sheaves} returns the following short exact sequence in cohomology,
\begin{align} \label{eq existence of Upsilon}
0 \longrightarrow \Hom_{X_B}(\Ww_{B,0} \otimes \Ww_{B,0},  H \otimes_k \Oo_X)^{\pm} \longrightarrow  \Hom&_{X_B}  (\Ww_{B,0} \otimes \Ww_{B,0}, \Oo_{X_B})^{\pm}  
\\
\longrightarrow & \Hom_{X_A}(\Ww_{A,0} \otimes \Ww_{A,0}, \Oo_{X_A})^{\pm} \longrightarrow 0.\nonumber
\end{align}
It follows that $(\pi_A \otimes \pi_A)^* \Phi_A$ determines a class $\left [ \Upsilon \right ]$ in
\begin{equation} \label{eq class of Upsilon}
\quotient{\Hom_{X_B}(\Ww_{B,0} \otimes \Ww_{B,0}, \Oo_{X_B})^{\pm}}{\sigma_0 \circ \Hom_{X_B}(\Ww_{B,0} \otimes \Ww_{B,0},  H \otimes_k \Oo_X)^{\pm},}
\end{equation}
where
\begin{align}  \label{eq up to the additive action of} 
\Hom_{X_B}(\Ww_{B,0} \otimes \Ww_{B,0},  H \otimes_k \Oo_X)^{\pm} & \cong \Hom_{X_B}(\Ww_{B,0}|_X \otimes \Ww_{B,0}|_X,  H \otimes_k \Oo_X)^{\pm}
\\
& \cong H \otimes_k \Hom_{X}(W_0 \otimes W_0,  \Oo_X)^{\pm}. \nonumber
\end{align}
Pick any representant $\Upsilon \in \Hom_{X_B}(\Ww_{B,0} \otimes \Ww_{B,0}, \Oo_{X_B})^{\pm}$ of the class $\left [ \Upsilon \right ]$ in \eqref{eq class of Upsilon} fixed by $(\pi_A \otimes \pi_A)^* \Phi_A$. We obviously have that $\Upsilon|_{X_A} = (\pi_A \otimes \pi_A)^* \Phi_A$. Therefore, the restriction to $X_A$ of the image under $\Upsilon$ of the subspace $\Ww_{B,0} \otimes \rho^{-1}(p_*\partial_{A, -1}(\Ww_{A, -1})) \subset \Ww_B \otimes \Ww_B$ vanishes,  
\[
\left . \Upsilon \left ( \Ww_{B,0} \otimes \rho^{-1}(p_*\partial_{A, -1}(\Ww_{A, -1})) \right ) \right |_{X_A} = (\pi_A \otimes \pi_A)^* \Phi_A \left ( \Ww_A \otimes \partial_{A, -1}(\Ww_{A, -1}) \right ) = 0.  
\]
It then follows that,
\begin{equation} \label{eq Upsilon of our subspace contained in I otimes Oo_X}
\Upsilon \left ( \Ww_{B,0} \otimes \rho^{-1}(p_*\partial_{A, -1}(\Ww_{A, -1})) \right ) \subset \sigma \left ( H \otimes_k \Oo_X \right ).
\end{equation}

Since $H^2 = 0$ one has that the intersection of $(H \otimes_k W_0) \otimes \Ww_{B,0}$ and $\Ww_{B,0} \otimes (H \otimes_k W_0)$ is $0 \cong (H \otimes_k W_0) \otimes (H \otimes_k W_0)$. Thanks to this, and recalling that $H \cdot \mM_A = 0$, one can then consider the subspaces of $\Ww_{B,0} \otimes \Ww_{B,0}$ 
\[
V_1 := (H \otimes_k W_0) \otimes \Ww_{B,0}  \cong (H \otimes_k W_0) \otimes \Ww_{A,0} \cong H \otimes_k (W_0 \otimes W_0), 
\]
and
\[
V_2 := \Ww_{B,0} \otimes (H \otimes_k W_0)  \cong \Ww_{A,0} \otimes (H \otimes_k W_0) \cong H \otimes_k (W_0 \otimes W_0).
\]
By construction $\Upsilon|_X = (\pi \otimes \pi)^*\phi$. Then, by continuity, we have that 
\begin{equation} \label{eq nu over V_i} 
\Upsilon|_{V_i} \cong \id_H \otimes (\pi \otimes \pi)^*\phi.
\end{equation}
It is a consequence of \eqref{eq nu over V_i} that
\begin{equation} \label{eq restriction of nu 1}
\Upsilon \left ( \Ww_{B,0} \otimes (H \otimes_k \partial_{-1}(W_{-1})) \right ) = 0
\end{equation}
since $\Ww_B \otimes (H \otimes_k \partial_{-1}(W_{-1})) \cong W_0 \otimes (H \otimes_k \partial_{-1}(W_{-1}))$. Also, as $(H \otimes_k W_0) \otimes \rho^{-1}(p_*\partial_{A, -1}(\Ww_{A, -1})) \cong (H \otimes_k W_0) \otimes \partial_{A, -1}(\Ww_{A, -1}) \cong (H \otimes_k W_0) \otimes \partial_{-1}(W_{-1})$, we have
\begin{equation} \label{eq restriction of nu 2}
\Upsilon \left ( (H \otimes_k W_0) \otimes \rho^{-1}(p_*\partial_{A, -1}(\Ww_{A, -1})) \right ) = 0. 
\end{equation}
It follows from \eqref{eq Upsilon of our subspace contained in I otimes Oo_X}, \eqref{eq restriction of nu 1} and \eqref{eq restriction of nu 2} that $\Upsilon$ applied to $\Ww_{B,0} \otimes \rho^{-1}(p_*\partial_{A, -1}(\Ww_{A, -1}))$ descends to the morphism of $\Oo_{X_B}$-modules
\[
\Upsilon_{\Ff} : \Ww_{A, 0} \otimes \Ff \to H \otimes_k \Oo_X.
\]
Consider $\Upsilon_{\Ff}^A: \Ww_{A, 0} \otimes \Ff^A \to H \otimes_k \Oo_X$ to be the morphism of $\Oo_{X_A}$-modules associated to $\Upsilon_{\Ff}$ under the equivalence mentioned above. Since we defined $F_1 = \Ff^A|_X$, let us set accordingly
\begin{equation} \label{eq definition of mu}
\mu:= \Upsilon_{\Ff}^A|_X : W_0 \otimes F_1 \to H \otimes_k \Oo_X.
\end{equation}
Recall that $\Upsilon$ is defined up to the additive action of \eqref{eq up to the additive action of}. We note that this action corresponds to the additive action of $\Hom_X(W_0 \otimes W_0, H \otimes_k \Oo_X)^{\pm} \circ (\id_{W_0} \otimes f)$ over $\mu$. It follows from \eqref{eq nu over V_i} that $\mu$ satisfies \eqref{eq mu and phi}. Also, as $\Upsilon = \pm \Upsilon \circ \theta_{\Ww_{B,0}}$ by construction, one naturally has that $\mu$ satisfies \eqref{eq mu and nu} as well. Note that for any other locally free resolution satisfying \eqref{eq cohomology vanishing for W_B}, we would obtain a $2$-extension in the same equivalence class as \eqref{eq getting 2-extension} and the corresponding class of morphism as \eqref{eq definition of mu}.

Thus, we have completed the construction of a morphism associated to the deformation functor $\Def_{(E,\phi)}^{\, +}$ and an small extension of Artin algebras $0 \to H \to B \stackrel{\tau}{\to} A \to 0$,  
\begin{equation} \label{eq obs morphism for Def^+} 
\Omega^{+}_{\tau}: \Def^{\, +}_{(E,\phi)}(A) \longrightarrow  H \otimes_k \HH^2 \left ( \Delta^{\, +, \bullet}_{(E,\phi)} \right ),  
\end{equation}
sending $(\Ee_A, \Phi_A, \gamma_A) \in \Def^{\, \pm}_{(E,\phi)}(A)$ to the point of $H \otimes_k \HH^2  \left ( \Delta^{\, +, \bullet}_{(E,\phi)} \right )$ given by the $2$-extension \eqref{eq getting 2-extension} and the class of morphisms given by \eqref{eq definition of mu}. Similarly, associated to $\Def^{\, -}_{(E,\phi)}$, we construct
\begin{equation} \label{eq obs morphism for Def^-} 
\Omega^{-}_{\tau}: \Def^{\, -}_{(E,\phi)}(A) \longrightarrow  H \otimes_k \HH^2 \left ( \Delta^{\, -, \bullet}_{(E,\phi)} \right ).
\end{equation}

We now see that these maps provide an obstruction theory for orthogonal and symplectic sheaves. Some results of obstruction theory of sheaves are needed and, for the reader's convenience, we include them instead of just cite them. We start by checking condition \textbf{O1}.

\begin{proposition} \label{pr O1}
Consider the small extension of Artin algebras $0 \to H \to B \stackrel{\tau}{\to} A \to 0$. Given $(\Ee_A, \Phi_A, \gamma_A) \in \Def^\pm_{(E,\phi)}(A)$, one has $\Omega^\pm_\tau(\Ee_A, \Phi_A, \gamma_A) = 0$ if and only if there exists $(\Ee_B, \Phi_B, \gamma_B) \in \Def^{\, \pm}_{(E,\phi)}(B)$ such that $(\Ee_B, \Phi_B, \gamma_B)|_{X_A} \cong (\Ee_A, \Phi_A, \gamma_A)$. 
\end{proposition}

\begin{proof}
From obstruction theory of sheaves, $(\Ee_A, \gamma_A)$ lifts to $(\Ee_B, \gamma_B)$ if and only if one can give an exact filler of \eqref{eq diagram to be filled}. It is a standard result of abelian categories (see \cite[Lemma 3.10]{nitsure} for instance) that exact fillers of \eqref{eq diagram to be filled} exist if and only if the short exact sequence \eqref{eq ses obs th} splits. In that case there exists a splitting morphism,
\[
s: p_*\partial_{A,-1}(\Ww_{A,-1}) \longrightarrow \Ff,
\]
whose composition with $\rho$ is the identity. Hence
\begin{equation} \label{eq description of the image of s}
s \left ( p_*\partial_{A,-1}(\Ww_{A,-1}) \right ) \cong \ker \pi_A.
\end{equation}
Fixing an splitting morphism, we can define $\Vv_B \subset \Ww_{B,0}$ as the preimage of $s \left ( p_*\partial_{A,-1}(\Ww_{A,-1}) \right )$ under the projection 
\[
\rho^{-1}\left ( p_*\partial_{A,-1}(\Ww_{A,-1}) \right ) \longrightarrow \Ff = \quotient{\rho^{-1}\left (p_* \partial_{A,-1}(\Ww_{A,-1}) \right )}{H \otimes_k \partial_{-1}(W_{-1})}.
\]
Then, set
\[
\Ee_{B} : = \quotient{\Ww_{B,0}}{\Vv_B}.
\]
One can easily check that this construction provides a coherent sheaf $\Ee_B$ over $X_B$ satisfying $\Ee_B|_{X_A} \cong \Ee_A$ and, furthermore, $\Ee_B$ is flat over $B$ (see for instance \cite[Lemma 3.14]{nitsure}). One trivially has that $\Ee_B|_X = \Ee_A|_X$, so we pick $\gamma_B$ to be $\gamma_A: \Ee_B|_X = \Ee_A|_X \stackrel{\cong}{\longrightarrow} E$.

Up to this point, we have just reproduced the classical theory of sheaves, seeing $(\Ee_A, \gamma_A)$ lifts to $(\Ee_B, \gamma_B)$ if and only if \eqref{eq ses obs th} splits. Pick $[\Upsilon]$ to be the class in \eqref{eq class of Upsilon} given by $(\pi_A \otimes \pi_A)^*\Phi_A$. If \eqref{eq ses obs th} splits, we have that a representant $\Upsilon$ of $[\Upsilon]$ defines $\Phi_B \in \Hom_{X_B}(\Ee_{B} \otimes \Ee_{B}, \Oo_{X_B})^{\pm}$ if and only if 
\begin{equation} \label{eq when Upsilon lifts}
\Upsilon \left ( \Ww_{B,0} \otimes \Vv_B \right ) = 0.
\end{equation}
It is a consequence of \eqref{eq Upsilon of our subspace contained in I otimes Oo_X}, \eqref{eq restriction of nu 1} and \eqref{eq restriction of nu 2} that \eqref{eq when Upsilon lifts} holds whenever 
\begin{equation} \label{eq when Upsilon' lifts}
\Upsilon_{\Ff} \left ( \Ww_{A,0} \otimes s \left ( p_* \partial_{A,-1}(\Ww_{A,-1})  \right ) \right ) = 0.
\end{equation}

It remains to show that \eqref{eq ses obs th} splits and \eqref{eq when Upsilon' lifts} holds for some representant of $\left [ \Upsilon \right ]$ in \eqref{eq class of Upsilon} if and only if the image of $\Omega^\pm_\tau(\Ee_A, \Phi_A, \gamma_A) = 0$. By the second statement of Proposition \ref{pr description HH^2}, the later is equivalent to the fact that the $2$-extension given in \eqref{eq getting 2-extension} has the form \eqref{eq split ses} and the class $[\mu]$ in $\Hom_X(W_0 \otimes F, H \otimes_k \Oo_X)^{\pm} / \Hom_X(W_0 \otimes W_0, H \otimes_k \Oo_X)^{\pm} \circ (\id_{W_0} \otimes f)$ defined in \eqref{eq definition of mu} satisfies \eqref{eq condition for mu to lift}. Note that \eqref{eq ses obs th} splits if and only if \eqref{eq ses obs th in X_A} splits, and further, \eqref{eq ses obs th in X_A} splits if and only if the $2$-extension \eqref{eq 2-ext obs th in X_A} splits. Since $\Ext^2_{X_A}(\Ee_A, H \otimes_k p_{A,*} E) \cong \Ext^2_X(E, H \otimes_k E)$, we have that \eqref{eq 2-ext obs th in X_A} splits if and only if \eqref{eq getting 2-extension} splits giving rise to a $2$-extension of the form \eqref{eq split ses}. In this case and recalling \eqref{eq description of the image of s}, the equation \eqref{eq when Upsilon' lifts} holds if and only if
\begin{equation} \label{eq when mu lifts}
\mu(W_0 \otimes \ker \pi) = 0.
\end{equation}
This is the case whenever the class $[\mu]$ satisfies \eqref{eq condition for mu to lift}, so $\Omega^\pm_\tau$ the proof is complete. 
\end{proof}

We check now that the morphisms $\Omega_\tau^\pm$ satisfy condition \textbf{O2}. Observe that any morphism of small extension decomposes into 
\begin{equation} \label{eq extension for O2 a}
\xymatrix{
0 \ar[r] & H \ar[r] \ar[d]_{\overline{h}} & B \ar[r]^{\tau} \ar[d]_{\overline{\beta}} & A \ar[r] \ar@{=}[d] & 0
\\
0 \ar[r] &\overline{H} :=  H/\ker \overline{h} \ar[r] & \overline{B} : = B/\ker \overline{h} \ar[r]^{\quad \quad \overline{\tau}} & A \ar[r] & 0,
}
\end{equation}
and
\begin{equation} \label{eq extension for O2 b}
\xymatrix{
0 \ar[r] & H \ar[r] \ar@{^(->}[d]_{h} & B \ar[r]^{\tau} \ar[d]_{\beta} & A \ar[r] \ar[d]^{\alpha} & 0
\\
0 \ar[r] & H' \ar[r] & B' \ar[r]^{\tau'} & A' \ar[r] & 0,
}
\end{equation}
where $h$ is an injection. Therefore, to check that \textbf{O2} holds, it is enough to check it for morphisms of small extensions of the form \eqref{eq extension for O2 a} and \eqref{eq extension for O2 b}. We start by the first type.

\begin{proposition} \label{pr O2 a}
Consider the morphism of small extensions \eqref{eq extension for O2 a}. Given $(\Ee_A, \Phi_A, \gamma_A) \in \Def^{\, \pm}_{(E,\phi)}(A)$, one has that 
\[
\Omega^\pm_{\overline{\tau}}(\Ee_A, \Phi_A, \gamma_A) = \left (\overline{h} \otimes \id_{\HH^2} \right) \circ \Omega^\pm_\tau(\Ee_A, \Phi_A, \gamma_A).
\]
\end{proposition}

\begin{proof}
Since one obviously has that $\overline{h}$ sends $\ker \overline{h}$ to $0$, observe that $\overline{h} \otimes \id_{\HH^2}$ applied to the $2$-extension \eqref{eq getting 2-extension} gives
\begin{equation} \label{eq 2-ext a 1}
0 \to \overline{H} \otimes_k E \to (F/\ker \overline{h} \otimes_k E) \to W_0 \to E \to 0.
\end{equation}
Pick now $\mu : W_0 \otimes F \to H \otimes_k \Oo_X$ given in \eqref{eq definition of mu} and note that 
\[
\left (\overline{h} \otimes_k \id_{\Oo_X} \right ) \circ \mu (W_0 \otimes (\ker \overline{h} \otimes_k E)) = 0.
\]
Then, $\left (\overline{h} \otimes_k \id_{\Oo_X} \right ) \circ \mu $ descends to 
\begin{equation} \label{eq mu a 1}
\overline{\mu} : W_0 \otimes (F/ (\ker \overline{h} \otimes_k E)) \to \overline{H} \otimes_k E,
\end{equation}
which is the image of $\mu$ under $\overline{h} \otimes_k \id_{\HH^2}$. We have seen that $\left (\overline{h} \otimes \id_{\HH^2} \right ) \circ \Omega^\pm_\tau(\Ee_A, \Phi_A, \gamma_A)$ is determined by \eqref{eq 2-ext a 1} and \eqref{eq mu a 1}.  

We now study $\Omega^\pm_{\overline{\tau}}(\Ee_A, \Phi_A, \gamma_A)$. Since $X_A \stackrel{p}{\hookrightarrow} X_B$ is the composition $X_A \stackrel{\overline{p}}{\hookrightarrow} X_{\overline{B}} \stackrel{p_{\, \overline{\beta}}}{\hookrightarrow} X_B$ one can consider in this case $\Ww_{\overline{B}, 0} := \Ww_{B,0} |_{X_{\overline{B}}}$ and the morphism $\overline{\rho} : \Ww_{\overline{B}, 0} \to \overline{p}_*\Ww_{A,0}$ being the restriction $\rho|_{X_{\overline{B}}}$. Pick also $\overline{\sigma} : \overline{H} \otimes_k W_0 \to \Ww_{\overline{B},0}$ corresponding to $\sigma$ on $X_{\overline{B}}$. Recall $\Ff$ from \eqref{eq def of Ff} and construct $\overline{\Ff}$ accordingly. It follows from the previous discussion that 
\[
\Ff|_{X_{\overline{B}}} = \overline{\Ff}.
\]
Observing that $\ker \overline{h} \otimes_k E$ is the kernel of $\Ff \to \overline{\Ff} = \Ff|_{X_{\overline{B}}}$, it then follows that
\begin{equation} \label{eq description of ol Ff}
\overline{\Ff}^A \cong \quotient{\Ff^A}{\ker \overline{h} \otimes_k E}.
\end{equation}

Up to here, we have been dealing with obstruction theory of sheaves. We now address the quadratic form. Observe that \eqref{eq existence of Upsilon} also holds for $X_{\overline{B}}$. Therefore $(\pi_A \otimes \pi_A)^*\Phi_A$ defines a class in the space \eqref{eq class of Upsilon} adapted to $X_{\overline{B}}$ from which can pick a representant $\overline{\Upsilon} \in \Hom_{X_{\overline{B}}}(\Ww_{\overline{B},0} \otimes \Ww_{\overline{B},0}, \Oo_{\overline{B}})^\pm$ that satisfies
\begin{equation} \label{eq description of ol Upsilon}
\Upsilon|_{X_{\overline{B}}} = \overline{\Upsilon}.    
\end{equation}
It is a direct consequence of \eqref{eq description of ol Ff} and \eqref{eq description of ol Upsilon} that $\Omega^\pm_{\overline{\tau}}(\Ee_A, \Phi_A, \gamma_A)$ is determined by \eqref{eq 2-ext a 1} and \eqref{eq mu a 1}. This concludes the proof.
\end{proof}

Since $h$ on \eqref{eq extension for O2 b} is injective, $H$ defines naturally a subspace of $H'$ so one can give (non-canonically) a decomposition of the vector spaces 
\begin{equation} \label{eq decomposition of H'}
H' = H \oplus H''
\end{equation}
and we can assume 
\begin{equation} \label{eq decomposition of h}
h = \id_{H} \oplus 0.
\end{equation}
Out of \eqref{eq extension for O2 b} and the (non-canonical) decomposition \eqref{eq decomposition of H'}, one can always construct a small extension $\overline{\tau}'$ and a morphism of small extensions making the following diagram commutative,
\begin{equation} \label{eq extension for O2 c}
\xymatrix{
0 \ar[r] & H \ar[r] \ar[d]_{h} & B \ar[r]^{\tau} \ar[d]_{\beta} & A \ar[r] \ar[d]^{\alpha} & 0
\\
0 \ar[r] & H \oplus H'' \ar[r] \ar[d]_{\overline{h}'} & B' \ar[r]^{\tau'} \ar[d]_{\overline{\beta}'} & A' \ar[r] \ar@{=}[d] & 0
\\
0 \ar[r] & H \ar[r] & \overline{B}':= B'/H'' \ar[r]^{\quad \quad \overline{\tau}'} & A' \ar[r] & 0,
}
\end{equation}
where $\overline{h}'$ and $\overline{\beta}'$ are the obvious projections. If we set $\beta':= \overline{\beta}' \circ \beta$ and note that $\overline{h}' \circ h = \id_H$, we obtain the morphism of small extensions 
\begin{equation} \label{eq extension for O2 d}
\xymatrix{
0 \ar[r] & H \ar[r] \ar@{=}[d] & B \ar[r]^{\tau} \ar[d]_{\beta'} & A \ar[r] \ar[d]^{\alpha} & 0
\\
0 \ar[r] & H \ar[r] & \overline{B}' \ar[r]^{\overline{\tau}'} & A' \ar[r] & 0.
}
\end{equation}
Let us consider the morphisms of schemes associated to the morphism of algebras appearing in \eqref{eq extension for O2 c} and \eqref{eq extension for O2 d}. Thanks to the commutativity of \eqref{eq extension for O2 c} one has the following commuting diagram of morphisms between schemes,
\[
\xymatrix{
X_{A'} \ar@{^(->}[dd]_{p'} \ar@{^(->}[rd]^{\overline{p}'} \ar[rr]^{p_\alpha} & & X_A \ar@{^(->}[dd]^{p}
\\
& X_{\overline{B}'} \ar@{^(->}[ld]^{p_{\overline{\beta}'}} \ar[rd]^{p_{\beta'}}   & 
\\
X_{B'} \ar[rr]_{p_\beta} & & X_B,
}
\]
whose right-upper subdiagram 
\begin{equation}  \label{eq Cartesian diagram for schemes}
\xymatrix{
X_{A'} \ar@{^(->}[d]_{\overline{p}'} \ar[r]^{p_\alpha}  & X_A \ar@{^(->}[d]^{p}
\\
X_{\overline{B}'}  \ar[r]_{p_{\beta'}} & X_B   
}
\end{equation}
is Cartesian.

Before checking $\mathbf{O2}$ restricted to morphisms of small extensions of the form \eqref{eq extension for O2 b}, we will study its compatibility with those of the form \eqref{eq extension for O2 d}.

\begin{proposition} \label{pr O2 c}
Consider the morphism of small extensions \eqref{eq extension for O2 d}. Given $(\Ee_A, \Phi_A, \gamma_A) \in \Def^{\, \pm}_{(E,\phi)}(A)$, one has that 
\[
\Omega^\pm_{\overline{\tau}'}\left ( p_\alpha^* (\Ee_A, \Phi_A, \gamma_A) \right ) = \Omega^\pm_\tau(\Ee_A, \Phi_A, \gamma_A)).
\]
\end{proposition}

\begin{proof}
We study $\Omega^\pm_{\overline{\tau}'}( p_{\alpha}^{\, *}(\Ee_A, \Phi_A, \gamma_A))$. Consider the locally free resolution $\Ww_A^\bullet \stackrel{\pi_A}{\longrightarrow} \Ee_A \to 0$ and take its pull-back $p_\alpha^*\Ww_A^\bullet \stackrel{p_\alpha^*\pi_A}{\longrightarrow} p_\alpha^*\Ee_A \to 0$ which is obviously a locally free resolution of $p_\alpha^*\Ee_A$. One can choose $\Ww_{B,0}$ and $\Ww_{\overline{B}',0} := p_{\beta'}^*\Ww_{B,0}$ satisfying in both cases the cohomology vanishing \eqref{eq cohomology vanishing for W_B} for $i>0$.

Recall that the morphism
\[
\overline{\rho}' : \Ww_{\overline{B}',0} = p_{\beta'}^*\Ww_{B,0} \longrightarrow \overline{p}'_*\Ww_{A',0} = \overline{p}'_*p_\alpha^* \Ww_{A,0},
\]
is given by the restriction to $X_{A'}$. Since $p_{\beta'}|_{X_{A'}} = p_\alpha$, $\overline{p}'|_{X_{A'}} = \id_{X_{A'}}$ and $p|_{X_A} = \id_{X_A}$, one has that 
\[
\left . \overline{p}'_*p_\alpha^* \Ww_{A,0} \right |_{X_{A'}} = \left . p_{\beta'}^* p_* \Ww_{A,0} \right |_{X_{A'}}.
\]
By the Cartesianity of \eqref{eq Cartesian diagram for schemes}, it follows that $\overline{p}'_*p_\alpha^* \Ww_{A,0}$ is supported over $X_{A'}$. By all of the above, one has that $\overline{\rho}' \cong p_{\beta'}^*\rho$ and this implies that
\begin{equation} \label{eq description of ol Ff'}
\overline{\Ff}' \cong p_{\beta'}^* \Ff.
\end{equation}
Hence $\overline{F}' := ( \overline{\Ff}')^{A'}|_X$ is isomorphic to $F$ and fits in the $2$-extension \eqref{eq getting 2-extension}.

We now move forward from the obstruction theory of sheaves. Recall that we chose $\Ww_{B,0}$ and $\Ww_{\overline{B}',0} := p_{\beta'}^*\Ww_{B,0}$ in such a way that both satisfy the cohomology vanishing \eqref{eq cohomology vanishing for W_B} for $i>0$. Let $\left [ \Upsilon \right ]$ be the class in \eqref{eq class of Upsilon} determined by $(\pi_A \otimes \pi_A)^*\Phi_A$ and note that $p_{\beta'}^*\Upsilon$ is a representant of the corresponding class determined to $(\pi_{A'} \otimes \pi_{A'})^*p_\alpha^*\Phi_A$. After this and \eqref{eq description of ol Ff'}, it follows that 
\[
\overline{\mu}': = \left . \left ( p_{\beta'}^* \Upsilon \right )^{A'}_{p_{\beta'}^*\Ff} \right |_X \cong \left .  p_{\alpha}^* \left (\Upsilon^{A}_{\Ff} \right ) \right |_X \cong \mu,
\]
where $\mu$ is defined in \eqref{eq definition of mu}. This concludes the proof, as $\mu$ and the $2$-extension \eqref{eq getting 2-extension} determine  $\Omega^\pm_{\tau}(\Ee_A, \Phi_A, \gamma_A)$ as well.
\end{proof}

We now check that the set of maps defined in \eqref{eq obs morphism for Def^+} and \eqref{eq obs morphism for Def^-} satisfy condition $\mathbf{O2}$ restricted to morphisms of small extensions of the form \eqref{eq extension for O2 b}. 

\begin{proposition} \label{pr O2 b}
Consider the morphism of small extensions \eqref{eq extension for O2 b}. Given $(\Ee_A, \Phi_A, \gamma_A) \in \Def^{\, \pm}_{(E,\phi)}(A)$, one has that 
\[
\Omega^\pm_{\tau'}\left ( p_\alpha^* (\Ee_A, \Phi_A, \gamma_A) \right ) = (h \otimes \id_{\HH^2}) \circ \Omega^\pm_\tau(\Ee_A, \Phi_A, \gamma_A)).
\]
\end{proposition}

\begin{proof}
We first describe $(h \otimes \id_{\HH^2}) \circ \Omega^\pm_\tau(\Ee_A, \Phi_A, \gamma_A)$, where we recall that $\Omega^\pm_\tau(\Ee_A, \Phi_A, \gamma_A)$ is determined by the $2$-extension \eqref{eq getting 2-extension} and $\mu$ given in \eqref{eq definition of mu}. Recall as well that $H'$ and $h$ decompose as indicated in \eqref{eq decomposition of H'} and \eqref{eq decomposition of h}. In that case, observe that $h \otimes \id_{\HH^2}$ applied to \eqref{eq getting 2-extension} gives
\begin{equation} \label{eq h circ zeta}
0 \to (H \otimes_k E) \oplus (H'' \otimes_k E) \to (F \otimes_k E) \oplus (H'' \otimes_k E) \to W_0 \to E \to 0.
\end{equation}
Also, note that the image of $\mu$ under $(h \otimes \id_{\HH^2})$ decomposes in two direct summands. The first summand is $\mu$ and the second is the zero element, which corresponds to $\id_{H''} \otimes_k \phi(\pi \otimes \id_E) : W_0 \otimes (H'' \otimes_k E) \to H'' \otimes_k \Oo_X$ after \eqref{eq mu and nu} and \eqref{eq condition for mu to lift}. Then, the image of $\mu$ under $h \otimes_k \id_{\HH^2}$ is 
\begin{equation} \label{eq h circ mu}
\mu \oplus \left( \id_{H''} \otimes_k \phi(\pi \otimes \id_E) \right ) : W_0 \otimes \left ( (F \otimes_k E) \oplus (H'' \otimes_k E) \right ) \longrightarrow (H \oplus H'') \otimes_k \Oo_X.
\end{equation}
 
Let us now study $\Omega^\pm_{\tau'}( p_{\alpha}^{\, *}(\Ee_A, \Phi_A, \gamma_A))$. We consider $\Ww_{A',i}$, $\Ww_{B',0}$ and the morphisms $\rho' : \Ww_{B',0} \to p'_*\Ww_{A',0}$ as we did in the beginning of this section. Set as well 
\begin{equation} \label{eq Ww_B' restricted to ol B'}
\Ww_{B',0} |_{X_{\overline{B}'}} = \Ww_{\overline{B}',0}.
\end{equation}
and 
\begin{equation} \label{eq rho' restricted to ol B'}
\rho' |_{X_{\overline{B}'}} = \overline{\rho}' : \Ww_{\overline{B}',0} \to \overline{p}'_*\Ww_{A',0} = p'_*\Ww_{A',0}|_{X_{\overline{B}'}}.
\end{equation}
Defining $\Ff'$ and $\overline{\Ff}'$ as in \eqref{eq def of Ff}, it follows from \eqref{eq Ww_B' restricted to ol B'} and \eqref{eq rho' restricted to ol B'} that 
\begin{equation} \label{eq Ff' restricted to ol B'}
\Ff' |_{X_{\overline{B}'}} = \overline{\Ff}'.
\end{equation}
From the description of $\Ff'$ that we obtain from \eqref{eq ses obs th} one can obtain the following short exact sequence of $\Oo_{X_{B'}}$-modules
\[
0 \longrightarrow H'' \otimes_k E \longrightarrow \Ff' \longrightarrow p_{\overline{\beta}',*} \overline{\Ff}' \longrightarrow 0,
\]
that gives rise to the short exact sequence of $\Oo_{X_{A'}}$-modules
\begin{equation} \label{eq ses for Ff' and ol Ff'}
0 \longrightarrow H'' \otimes_k p_{A', *}E \longrightarrow (\Ff')^{A'} \longrightarrow p_{\overline{\beta}',*} (\overline{\Ff}')^{A'} \longrightarrow 0.
\end{equation}
We see that \eqref{eq Ff' restricted to ol B'} provides naturally a splitting of \eqref{eq ses for Ff' and ol Ff'}, so
\begin{equation} \label{eq decomposition of Ff'}
(\Ff')^{A'} \cong p_{\overline{\beta}',*} (\overline{\Ff}')^{A'} \oplus \left ( H'' \otimes_k p_{A', *}E \right ).
\end{equation}
Hence, the $2$-extension determined by $\Omega^\pm_{\tau'}( p_{\alpha}^{\, *}(\Ee_A, \Phi_A, \gamma_A))$ is
\begin{equation} \label{eq zeta'}
0 \to (H \otimes_k E) \oplus (H'' \otimes_k E) \to \left (\overline{F}' \otimes_k E \right ) \oplus (H'' \otimes_k E) \to W_0' \to E \to 0,
\end{equation}
where $\overline{F}'$ denotes $(\overline{\Ff}')^{A'}|_X$ and $W_0'$ is the restriction to $X$ of $\Ww_{A',0}$.

Only at this point, we find ourselves in a position to give a step forward from the classical case of sheaves. One has that \eqref{eq existence of Upsilon} also holds over $X_{B'}$ and $X_{\overline{B}'}$. Therefore $(\pi_A \otimes \pi_A)^*\Phi_A$ defines classes in the corresponding spaces \eqref{eq class of Upsilon} defined over $X_{B'}$ and $X_{\overline{B}'}$. Note also that one can choose representants $\Upsilon' \in \Hom_{X_{B'}}(\Ww_{B',0} \otimes \Ww_{B',0}, \Oo_{B'})^\pm$ and $\overline{\Upsilon}' \in \Hom_{X_{\overline{B}'}}(\Ww_{\overline{B}',0} \otimes \Ww_{\overline{B}',0}, \Oo_{\overline{B}'})^\pm$ satisfying
\begin{equation} \label{eq description of ol Upsilon' a}
\Upsilon'|_{X_{\overline{B}'}} = \overline{\Upsilon}'.
\end{equation}
Denote by $\mu'$ and $\overline{\mu}'$ the maps defined in \eqref{eq definition of mu} out of $\Upsilon'$ and $\overline{\Upsilon}'$. It then follows from \eqref{eq nu over V_i}, \eqref{eq decomposition of Ff'} and \eqref{eq description of ol Upsilon' a} that
\begin{equation} \label{eq mu'}
\mu' = \overline{\mu}' \oplus \left( \id_{H''} \otimes_k \phi(\pi \otimes \id_E) \right ).
\end{equation}

The result follows from Proposition \ref{pr O2 c} after comparing \eqref{eq h circ zeta} with \eqref{eq zeta'} and \eqref{eq h circ mu} with \eqref{eq mu'}.
\end{proof}

The following summarizes all the previous results in this section.

\begin{theorem}
Let $(E,\phi)$ be an orthogonal ({\it resp.} symplectic) sheaf over the projective scheme $X$. Then the deformation functor $\Def^{\, +}_{(E,\phi)}$ ({\it resp.} symplectic) admits an obstruction theory with vector space
\[
\Obs\left ( \Def^{\, \pm}_{(E,\phi)} \right) = \HH^2 \left ( \Delta^{\pm,\bullet}_{(E,\phi)} \right ).
\]
\[
\]
Therefore, $\Def^{\, \pm}_{(E,\phi)}$ are formally smooth when $\HH^2 \left ( \Delta^{\pm,\bullet}_{(E,\phi)} \right ) = 0$. 
\end{theorem}

\begin{proof}
The theorem is a consequence of Propositions \ref{pr O1}, \ref{pr O2 a} and \ref{pr O2 b}.
\end{proof}

\end{document}